\documentclass[11pt]{articlefederico}
\usepackage{graphicx}
\usepackage{amsfonts }
\usepackage{amsmath}
\usepackage{fullpage}
\usepackage{amssymb}
\usepackage{amsthm}
\usepackage{tikz}
\usepackage{lipsum}
\usepackage{mathrsfs}
\usepackage{hyperref}
\usepackage{float}
\usepackage{multicol}

\definecolor{orange}{HTML}{FF7F00}  
\definecolor{pink}{HTML}{FF007F}  
\definecolor{violet}{HTML}{7F00FF}  
\definecolor{azure}{HTML}{007FFF}  
\definecolor{springgreen}{HTML}{00FF7F}  
\definecolor{chartreuse}{HTML}{7FFF00}  

\usepackage{subcaption}

\def\myomega{\lambda}

\theoremstyle{definition}
\newtheorem{theorem}{Theorem}[section]

\newtheorem{cor}[theorem]{Corollary}
\newtheorem{lem}[theorem]{Lemma}

\newtheorem{prop}[theorem]{Proposition}
\newtheorem{defi}[theorem]{Definition}

\newtheorem{ex}[theorem]{Example}
\newtheorem{rem}[theorem]{Remark}

\makeatletter
\newtheorem*{rep@theorem}{\rep@title}
\newcommand{\newreptheorem}[2]{%
\newenvironment{rep#1}[1]{%
 \def\rep@title{#2 \ref{##1}}%
 \begin{rep@theorem}}%
 {\end{rep@theorem}}}
\makeatother

\newcommand*{\medcap}{\mathbin{\scalebox{1.5}{\ensuremath{\cap}}}}%

\newreptheorem{theorem}{Theorem}

\renewcommand{\AA}{{\cal A}}
\newcommand{\Z}{{\cal Z}}

\newcommand{\R}{{\mathscr{R}}}

\newcommand{\e}{\overline{e}}
\newcommand{\x}{\mathbf{x}}
\newcommand{\1}{\mathbf{1}}

\newcommand{\HH}{{\cal H}}
\newcommand{\ZZ}{\mathbb{Z}}
\newcommand{\RR}{{\mathbb R}}
\newcommand{\df}{\textrm{Def}}
\newcommand{\nf}{\textrm{Nef}}

\newcommand{\SF}{\textrm{SF}}

\newcommand{\cone}{\textrm{cone}}
\newcommand{\conv}{\textrm{conv}}
\renewcommand{\span}{\textrm{span}}
\newcommand{\PL}{\textrm{PL}}

\begin{document}
\title{\textsf{Coxeter submodular functions and \\
deformations of Coxeter permutohedra}}
\author{
\textsf{Federico Ardila\footnote{\noindent \textsf{San Francisco State University, Simons Institute for the Theory of Computing; U. de Los Andes; federico@sfsu.edu.}}}\\
\and \textsf{Federico Castillo\footnote{\noindent \textsf{University of Kansas, fcastillo@ku.edu}}} \\
\and \textsf{Christopher Eur\footnote{\noindent \textsf{University of California, Berkeley, ceur@math.berkeley.edu}}} \\
\and \textsf{Alexander Postnikov\footnote{\noindent \textsf{Massachusetts Institute of Technology, apost@math.mit.edu
\newline 
Partially supported by National Science Foundation DMS Awards 1440140 (MSRI), 1600609 + 1855610 (FA), 1764370 (AP), and the Simons Foundation (FA).}}}}

\date{}

\maketitle


\begin{abstract}
We describe the cone of deformations of a Coxeter permutohedron, or equivalently, the nef cone of the toric variety associated to a Coxeter complex. This class of polytopes contains important families such as weight polytopes, signed graphic zonotopes, Coxeter matroids, root cones, and Coxeter associahedra.
Our description extends the known correspondence between generalized permutohedra, polymatroids, and submodular functions to any finite reflection group.
\end{abstract}

\section{\textsf{Introduction}}

The permutohedron $\Pi_n$ is the convex hull of the $n!$ permutations of $(1, \ldots, n)$ in $\RR^n$. This  polytopal model for the symmetric group $S_n$ appears in numerous combinatorial, algebraic, and geometric settings. 
There are two natural generalizations: 

1.\ Reflection groups: Instead of $S_n$, we may consider any finite real reflection group $W$ with corresponding (not necessarily crystallographic) root system $\Phi \subset V$. This group is similarly modeled by the $\Phi$-permutohedron, which is the convex hull of the $W$-orbit of a generic point in $V$. Most geometric and representation theoretic properties of the permutohedron extend to this setting. 

2.\ Deformations: We may deform the polytope by moving its faces while preserving their directions. The resulting family of \emph{generalized permutohedra} or \emph{polymatroids} is special enough to feature a rich combinatorial, algebraic, and geometric structure, and flexible enough to contain polytopes of interest in numerous different contexts.

\smallskip

The goal of this paper is to describe the \emph{deformations of $\Phi$-permutohedra} or \emph{$\Phi$-polymatroids}, thus generalizing these two directions simultaneously. We have two motivations:

\smallskip

%
$\bullet$ \emph{Coxeter combinatorics} recognizes that many classical combinatorial constructions are intimately related to the symmetric group, and have natural generalizations to the setting of reflection groups. There are natural Coxeter analogs of compositions, graphs, matroids, posets, and clusters, all of which have polyhedra modeling them: weight polytopes \cite{FultonHarris}, signed graphic zonotopes \cite{Zaslavskysigned}, Coxeter matroids \cite{Coxetermatroids, GelfandSerganova}, root cones \cite{Reiner, Stembridge}
, and Coxeter associahedra \cite{HohlwegLangeThomas}. We observe that these are all deformations of Coxeter permutohedra.

\smallskip

$\bullet$ The \emph{Coxeter permutohedral variety} $X_\Phi$ is the toric variety associated to a crystallographic Coxeter arrangement ${\mathcal A}_\Phi$. The various embeddings of $X_\Phi$ into projective spaces give rise to the \emph{nef cone}, a key object in the toric minimal model program. The nef cone of $X_\Phi$ can be identified with the cone of possible deformations of the $\Phi$-permutohedron.

\medskip
Let us now summarize our main results on deformations of Coxeter permutahedra. The necessary definitions and detailed statements are presented in the upcoming sections.

%
%

\medskip
A central result about generalized permutohedra in $\RR^n$  is that they are in bijection with the functions $f: 2^{[n]} \rightarrow \RR$ that satisfy the \emph{submodular inequalities} $f(A) + f(B) \geq f(A \cup B) + f(A\cap B)$. Thus the field of submodular optimization is essentially a study of this family of polytopes.

Our main theorem extends this to all finite reflection groups. Let $\Phi$ be a finite root system with Weyl group\footnote{Some authors only associate the terminology ``Weyl group" to \emph{crystallographic} root systems.  In this paper, we will not assume root systems to be crystallographic unless explicitly stated otherwise.} $W$ and let $\R = W\{\myomega_1, \ldots, \myomega_d\}$ be the union of the $W$-orbits of the fundamental weights $\myomega_1, \ldots, \myomega_d$. Let $A$ be the Cartan matrix.

\newtheorem*{thm:main}{Theorem \ref{thm:main}}
\begin{thm:main}
The deformations of the $\Phi$-permutohedron are in bijection with the functions $h: \R \rightarrow \RR$ that satisfy the \emph{$\Phi$-submodular inequalities}:

\begin{equation}\label{ineq:localsubmodintro}
h(w  \myomega_i) +  h(ws_i  \myomega_i) \geq 
  \sum_{j \in N(i)}- 
  A_{ji} \,
  h(w \myomega_j) 
\end{equation}
for every element $w \in W$, every simple reflection $s_i$, and corresponding fundamental weight $\myomega_i$. Here $N(i)$ denotes the set of neighbors of $i$ in the Dynkin diagram, not including $i$ itself.
 \end{thm:main}

These inequalities are very sparse: The right hand side of the $\Phi$-submodular inequality has $1, 2,$ or $3$ non-zero terms, depending on the number of neighbors of $i$ in the Dynkin diagram of $\Phi$. For $\Phi = A_{n-1}$ we get precisely the classic family of submodular functions.  For $\Phi = B_n$ and $\Phi = C_n$ we get precisely Fujishige's notion of bisubmodular functions. More generally, we expect $\Phi$-submodular functions to be useful in  combinatorial optimization problems with an underlying symmetry of type $\Phi$.

We prove that the inequalities \eqref{ineq:localsubmodintro} are precisely the facets of the cone $\SF_\Phi$ of $\Phi$-submodular functions. This allows us to enumerate them.  Again, let $N(i)$ be the set of neighbors of $i$ in the Dynkin diagram, and for each subset $I\subseteq [d]$ let $W_I$ denote the parabolic subgroup of $W$ generated by $\{s_i\}_{i\in I}$.

\newtheorem*{thm:main2}{Theorems \ref{thm:facet} and \ref{thm:count}}
\begin{thm:main2}
Each inequality \eqref{ineq:localsubmodintro} associated with a pair $(w,s_i)$, for $w\in W$ and $i\in[d]$, gives a facet of the $\Phi$-submodular cone $\SF_\Phi$.  Two pairs $(w,s_i)$ and $(w',s_{i'})$ define the same facet if and only if $i=i'$ and $w^{-1}w'\in W_{[d]- N(i)}$.  In particular, the number of facets of $\SF_\Phi$ equals
$$
\sum_{i=1}^d {|W|\over |W_{[d] - N(i)}|}.
$$
\end{thm:main2}

On the other hand, the rays of the $\Phi$-submodular cone seem to be very difficult to describe, even when $\Phi = A_{n-1}$.

\medskip
We completely describe an important slice of $\SF_\Phi$: the \emph{symmetric $\Phi$ submodular cone} consisting of the $\Phi$-submodular functions that are invariant under the natural action of the Weyl group $W$. These functions correspond to the $W$-symmetric $\Phi$-generalized permutahedra; 
when $\Phi$ is crystallographic, this family includes the \emph{weight polytopes} that  parameterize the irreducible representations of the associated Lie algebra.  By identifying a symmetric $\Phi$-submodular function $h$ with the vector $(h(\myomega_1), \ldots, h(\myomega_d))\in \RR^d$ of its values on the fundamental weights, we may think of the symmetric $\Phi$-submodular cone $\SF_{\Phi}^{\operatorname{sym}}$ as living in $\RR^d$.  We obtain directly from Theorem \ref{thm:main} the following description of $\SF_{\Phi}^{\operatorname{sym}}$.

\newtheorem*{prop:sym}{Proposition \ref{prop:sym}}
\begin{prop:sym}
The \emph{symmetric $\Phi$-submodular cone} $\SF_\Phi^{\textrm{sym}}$ is the simplicial cone generated by the rows of the inverse Cartan matrix of $\Phi$.
\end{prop:sym}

We conclude by characterizing which weight polytopes are indeformable; or equivalently, which rays of the symmetric submodular cone $\SF_\Phi^{\textrm{sym}}$ are also rays of the submodular cone $\SF_\Phi$. 

\newtheorem*{thm:orbit}{Theorem \ref{thm:orbit}}
\begin{thm:orbit}
Let $\Phi$ be a crystallographic root system. 
The rays of the $\Phi$-submodular cone $\SF_\Phi$ that are symmetric correspond to the weight polytopes of the fundamental weights $\myomega_i$ where $i$ is a vertex of the Dynkin diagram whose incident edges are unlabelled.
\end{thm:orbit}

The paper is organized as follows. Sections \ref{sec:polytopes} reviews some preliminaries on polytopes and their deformations. Section \ref{sec:coxeter} reviews some basic facts about root systems, reflection groups, and Coxeter complexes. Section \ref{sec:examples} introduces Coxeter permutohedra and some of their important deformations.
In Section \ref{sec:Phisubmodularcone} we describe the $\Phi$-submodular cone $\SF_\Phi$, which parameterizes the deformations of the $\Phi$-permutohedron.
Section \ref{sec:symmetric} studies weight polytopes: the deformations of the $\Phi$-permutohedron that are invariant under the action of the Weyl group $W_\Phi$. The fundamental weight polytopes are especially important; they correspond to the $W$-symmetric $\Phi$-submodular functions, which are given by the inverse Cartan matrix. Section \ref{sec:facets} describes and enumerates the facets of the $\Phi$-submodular cone, while Section \ref{sec:rays} describes some of its rays. We conclude with some future research directions in Section \ref{sec:future}.

%


\section{\textsf{Polytopes and their deformations}}\label{sec:polytopes}

In this section we review the relationship between polytopes, their support functions, and their deformation cones. For a more detailed treatment and proofs, see for example \cite{toric} or \cite{triangulations}.

\subsection{\textsf{Polytopes and their support functions}}

Let $U$ and $V$ be two real vector space of finite dimension $d$ in duality via a perfect bilinear form $\langle \cdot,\cdot\rangle: U\times V \longrightarrow \RR$. 
A \emph{polyhedron} $P\subset V$ is an intersection of finitely many half-spaces; it is a \emph{polytope} if it is bounded.  
We will regard each vector $u \in U$ as a linear functional on $V$, which gives rise to the \emph{$u$-maximal face}
\[
P_u := \{v\in P: \langle u,v\rangle =\max_{x\in P} \langle u, x \rangle\}
\]
whenever $\max_{x\in P}\langle u , x\rangle$ is finite.  

Let $\Sigma_P$ be the \emph{(outer) normal fan} in $U$. For each $\ell$-codimensional face $F$ of $P$, the normal fan $\Sigma_P$ has a dual $\ell$-dimensional face
\[
\Sigma_P(F) = \{u \in U \, : \, P_u = F\}.
\]
The support $|\Sigma_P|$ of $\Sigma_P$ is the union of its faces. It equals $U$ if $P$ is a polytope.

A polyhedron $P$ is \emph{simple} if each vertex $v\in P$ is contained in exactly $d$ facets, or equivalently if every cone in $\Sigma_P$ is \emph{simplicial} in that its generating rays are linearly independent.  Each relative interior of a cone in a fan $\Sigma$ is called an \emph{open face}.  Denote by $\Sigma(\ell)$ the set of $\ell$-dimensional cones of $\Sigma$. We call the elements of $\Sigma(d)$ \emph{chambers} and the elements of $\Sigma(d-1)$ \emph{walls}; they are the full-dimensional and 1-codimensional faces of $\Sigma$, respectively. 

\medskip
All fans we consider in this paper will be normal fans $\Sigma_P$ of polyhedra $P$, so from now on we will assume that every fan $\Sigma\subset U$ has convex support.  We say that the fan $\Sigma$ is \emph{complete} if $|\Sigma| = U$ and \emph{projective} if $\Sigma = \Sigma_P$ for some polytope $P$.

\medskip
Given a fan $\Sigma \subset U$, we denote the space of \emph{continuous piecewise linear functions on $\Sigma$} by
\[
\operatorname{PL}(\Sigma) := \{f: |\Sigma| \to \RR \ | \ f \textnormal{ linear on each cone of $\Sigma$ and continuous}\}.
\]
It is a finite-dimensional vector space, since a piecewise linear function on $\Sigma$ is completely determined by its restriction to the rays of $\Sigma$. 

The \emph{support function} of a polyhedron $P$ is the element $h_P \in \operatorname{PL}(\Sigma_P)$ defined by 
\begin{equation}\label{eq:h_P}
h_P(u):=\max_{v\in P} \langle u,v\rangle \qquad \textrm{ for } u \in |\Sigma_P|.
\end{equation}
Notice that we can recover $P$ from $h_P$ uniquely by 
\[
P = \{v\in V: \langle u, v \rangle\leq h_P(u) \ \textrm{ for all } u\in |\Sigma_P|\},
\]
so a polyhedron and its support function uniquely determine each other.

Notice that the translation $P+v$ of a polyhedron $P$ has support function $h_{P+v}=h_P+h_{\{v\}}$, where $h_{\{v\}}$ is the linear functional $\langle \cdot,v\rangle$ (restricted to $|\Sigma_P|$).  Therefore translating a polytope $P$ is equivalent to adding a global linear functional to its support function $h_P$.

\medskip
We say two polyhedra $P,Q$ are \emph{normally equivalent} (or \emph{strongly combinatorially equivalent}) if $\Sigma_P = \Sigma_Q$.  It two fans $\Sigma$ and $\Sigma'$ have the same support, we say $\Sigma$ \emph{coarsens} $\Sigma'$ (or equivalently $\Sigma'$ \emph{refines} $\Sigma$) if each cone of $\Sigma$ is a union of cones in $\Sigma'$ (or equivalently, each cone of $\Sigma'$ is a subset of a cone of $\Sigma$).  We denote this relation by $\Sigma\preceq \Sigma'$.

\subsection{\textsf{Deformations of polytopes}}

While we will be primarily interested in deformations of polytopes, we first define them for polyhedra in general.  Let $P$ be a polyhedron.

\begin{defi}\label{defi:defo}
A polyhedron $Q$ is a \emph{deformation of $P$} if the normal fan $\Sigma_Q$ is a coarsening of the normal fan $\Sigma_P$.
\end{defi}

When $P$ is a simple polytope, it is shown in \cite[Theorem 15.3]{faces} that we may think of the deformations of $P$ equivalently as being obtained by any of the following three procedures:

$\bullet$ moving the vertices of $P$ while preserving the direction of each edge, or

$\bullet$ changing the edge lengths of $P$ while preserving the direction of each edge, or

$\bullet$ moving the facets of $P$ while preserving their directions, without allowing a facet to move past a vertex.

\begin{figure}[h]
\centering
\includegraphics[scale=.4]{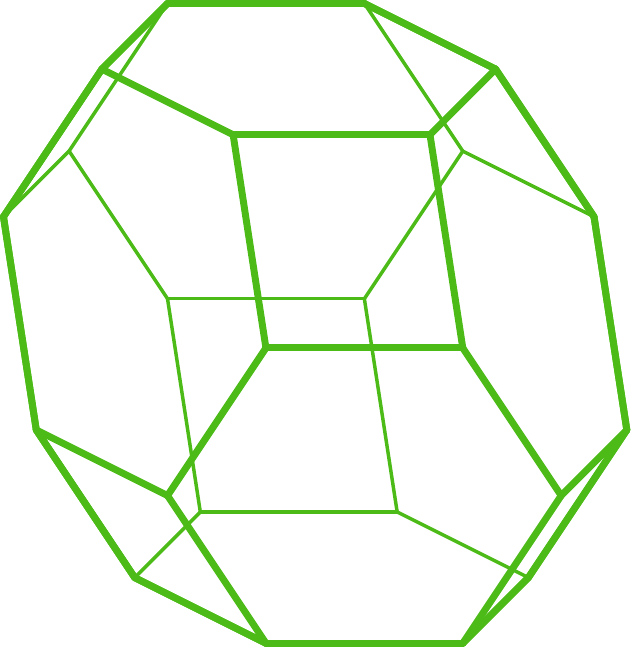} \qquad \qquad 
\includegraphics[scale=.5]{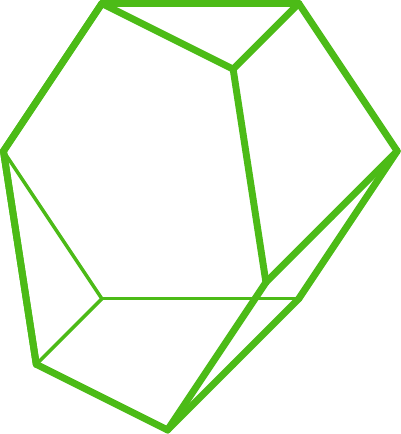}  
\caption{The standard 3-permutohedron and one of its deformations.\label{f:genperm}}
\end{figure}

\begin{rem}
More generally, a \emph{parallel redrawing} of a graph embedded into a linear space is a way to reposition its vertices so that the edges of the deformed graph are parallel to edges of the original. 
Parallel redrawings play an important role in Goresky-Kottwitz-MacPherson's theory \cite{GKM} of equivariant cohomology of GKM manifolds.  See \cite{BGH} for their implications for polytopes, \cite{FS} for applications to matroids, and \cite{PS} for a combinatorial formula for the dimension of the space of parallel redrawings of a generic embedding of a graph.
\end{rem}

By allowing certain facet directions to be unbounded in this deformation process, we obtain a larger family of polyhedra:

\begin{defi}
A polyhedron $Q$ is an \emph{extended deformation of $P$} if the normal fan $\Sigma_Q$ coarsens a convex subfan of $\Sigma_P$.
\end{defi}

In other words, an extended deformation $Q$ of a polyhedron $P$ is a deformation of a polyhedron $P'$ where $P' = \{v\in V : \langle v, u \rangle \leq h_P(u) \ \textrm{ for all } u\in |\Sigma'|\}$ for some convex subfan $\Sigma'$ of $\Sigma_P$.  Deformations are extended deformations with $\Sigma' = \Sigma_P$.

\medskip
For polytopes, Minkowski sums provide yet another way of thinking about deformations.  The \emph{Minkowski sum} of two polytopes $Q$ and $R$  in the same vector space $V$ is the polytope
\[
Q+R := \{q+r \, : \, q \in Q, r \in R\}.
\]
The support function of $Q+R$ is 
\[
h_{Q+R} = h_Q + h_R
\]
and the normal fan $\Sigma_{Q+R}$ is the coarsest common refinement of the normal fans $\Sigma_Q$ and $\Sigma_R$ \cite[Proposition 1.2]{BFS90}.  Therefore $Q$ is a deformation of $Q+R$. The next result shows that this is, up to scaling, the only source of deformations. For this reason, deformations of polytopes are also often called \emph{weak Minkowski summands}.

\begin{theorem}[Shephard \cite{grunbaum}]
Let $P$ and $Q$ be polytopes, then $Q$ is a deformation of $P$ if and only if there exist a polytope $R$ and a scalar $k >0$ such that $Q+R=k P$. 
\end{theorem}



\subsection{\textsf{Deformations of zonotopes}}

Let $\AA = \{v_1,\ldots,v_m\} \subset V$ be a set of vectors and let 
$\mathcal H = \{H_1,\ldots, H_m\}$ be the corresponding hyperplane arrangement in $U$ given by the hyperplanes $H_i=\{u\in U: \langle u,v_i\rangle=0\}$ for $1 \leq i \leq m$.  The hyperplane arrangement $\mathcal H$ then determines a fan $\Sigma_{\mathcal H}$ whose maximal cones are the closures of the connected components of the arrangement complement.

\begin{defi}
Let $\AA = \{v_1,\ldots,v_m\} \subset V$.  The \emph{zonotope} of $\AA$ is the Minkowski sum 
\[
{\Z}(\AA):=[0,v_1]+\cdots+[0,v_m].
\]
\end{defi}

The relationship between Minkowski sums and coarsening of fans imply that the normal fan of the zonotope $\Z(\AA)$ is equal to $\Sigma_{\mathcal H}$.  We can describe the (extended) deformations of $\mathcal Z(\AA)$ easily as follows.  

\begin{prop}\label{prop:zonoedges}
Let $\AA$ be a finite set of vectors in $V$.  A polyhedron $P$ is an extended deformation of $\mathcal Z(\mathcal A)$ if and only if every face affinely spans a parallel translate of $\operatorname{span}_\RR(S)$ for some $S\subseteq \mathcal A$.  In particular, a polytope is a deformation of the zonotope $\Z(\AA)$ if and only if every edge is parallel to some vector in $\AA$.
\end{prop}

\begin{proof}
We start with two easy observations.  First, if two cones $\sigma \subseteq \sigma'$  have the same dimension, then $\span_\RR(\sigma) = \span_\RR(\sigma')$.
Second, if $\sigma\in \Sigma_{\mathcal H}$, then 
$\operatorname{span}_\RR(\sigma) = \bigcap_{i\in S} H_i$ for some $S\subseteq \mathcal A$.

Now, let $P$ be an extended deformation of $\Z(\AA)$ and $F$ a face of $P$.  Then since $\Sigma_P$ coarsens a convex subfan of $\Sigma_{\mathcal H}$, the cone $\Sigma_P(F)$ has the same $\RR$-span as $\operatorname{span}_\RR(\sigma)$ for some $\sigma\in \Sigma_{\mathcal H}$.  This implies that the affine span of $F$ is a parallel translate of $\operatorname{span}_\RR(S)$ for some $S\subseteq \mathcal A$.  

Conversely, assume every face of $P$ satisfies the given condition. Then the fan $\Sigma_P$ has convex support and for each maximal cone $\sigma\in\Sigma_P$, one has $\operatorname{span}_\RR(\sigma) = \bigcap_{i\in S} H_i$ for some $S\subseteq \mathcal A$.  
We may assume that $\Sigma_P$ is full dimensional: If it is not, then it is contained in a non-proper linear subspace $L = \bigcap_{j\in T} H_j$ of $U$ for some $T\subset \mathcal A$. Equivalently $P$ has a non-trivial lineality space $L^\perp := \{v\in V : \langle v, u \rangle = 0 \ \textrm{ for all } u \in L\}$; this is the maximal subspace of $V$ such that $P + L^\perp = P$. Thus we may replace $U$ with $L$, $V$ with $V/L^\perp$, and $P$ with $P/L^\perp$. 
Now, since $\Sigma_P$ is full dimensional in $U$, all of its walls are contained in some hyperplane $H_i$.  Collecting these hyperplanes gives a subarrangement $\mathcal H'$ of $\mathcal H$, whose fan $\Sigma_{\mathcal H'}$ restricted to $|\Sigma_P|$ is exactly $\Sigma_P$, as desired.  
\end{proof}

\begin{cor}\label{cor:zonofaces}
Let $\AA$ be a finite set of vectors in $V$. If $P$ is a(n extended) deformation of 
$\mathcal Z(\mathcal A)$ then any face of $P$ is a(n extended) deformation of $\mathcal Z(\mathcal A)$.
\end{cor}

\subsection{\textsf{Deformation cones}}

Let $P$ be a polyhedron in $V$ and $\Sigma = \Sigma_P$ be its normal fan in $U$.  In this section we will assume  $\Sigma$ is full dimensional. This results in no loss of generality, as shown in the proof of Proposition \ref{prop:zonoedges}.
%

\medskip
For each deformation $Q$ of $P$, the normal fan $\Sigma_Q$ coarsens $\Sigma$, and hence the support function $h_Q$ defined in \eqref{eq:h_P} is piecewise-linear on $\Sigma$.  Thus, by identifying $Q$ with its support function $h_Q$, we see that the deformations of $P$ form a cone.

\begin{theorem}\label{th:defcone} \cite[Theorems 6.1.5--6.1.7]{toric}.
Let $P$ be a polyhedron in $V$ and $\Sigma = \Sigma_P$ be its normal fan. 
The \emph{deformation cone} of $P$ (or of $\Sigma$) is
\begin{eqnarray*}
\df(P) = \df(\Sigma) & := & \{h_Q \, | \, Q \textrm{ is a deformation of } P\} \\
&=& \{h_Q\in \operatorname{PL}(\Sigma) \ | \ \Sigma_Q\preceq \Sigma\} \\
&=& \{h \in \operatorname{PL}(\Sigma) \ | \ h \textrm{ is convex}\}. 
\end{eqnarray*}
\end{theorem}

\begin{rem}
For each ray $\rho \in \Sigma(1)$ let $u_\rho$ be a vector in the direction of $\rho$. When $\Sigma$ is a rational fan, we let $u_\rho$ be the first lattice point on the ray $\rho$. Let $\mathcal R = \{u_{\rho} \, : \, \rho \in \Sigma(1)\}$. A piecewise linear function on $\Sigma$ is determined by its values on each $u_\rho$, so we may regard it as a function $h: \mathcal R \rightarrow \RR$. Therefore we can think of $\operatorname{PL}(\Sigma)$ as a subspace of  $\RR^{\mathcal R}$. We have
\[
 \operatorname{PL}(\Sigma) \cong  \RR^{\mathcal R} \qquad \textrm{ if $\Sigma$ is simplicial}
 \]
since in this case the values $h(u_\rho)$ may be chosen arbitrarily.
\end{rem}

It is known that $\operatorname{Def}(\Sigma)$ is a polyhedral cone of dimension $\dim_\RR \operatorname{PL}(\Sigma)$. There is a \emph{wall-crossing criterion}, consisting of finitely many linear inequalities, to test whether a piecewise linear function $h \in \PL(\Sigma) \subseteq \RR^R$ is convex. We now review two versions of this  criterion: a general one in Section \ref{sec:wallcrossing}, and a simpler one that holds for simple polytopes (or simplicial fans) in Section \ref{sec:Batyrev}.

\subsubsection{\textsf{The wall crossing criterion} \label{sec:wallcrossing}}

\begin{defi}\label{defi:wallcross} \emph{(Wall-crossing inequalities)}
Let $\tau\in \Sigma(d-1)$ be a wall separating two chambers $\sigma$ and $\sigma'$ of $\Sigma$.  
Choose any $d-1$ linearly independent rays $\rho_1, \ldots, \rho_{d-1}$ of $\tau$ and any two rays $\rho, \rho'$ of $\sigma,\sigma'$, respectively, that are not in $\tau$. Up to scaling, there is a unique linear dependence of the form
\begin{equation}\label{eq:wallcross}
\displaystyle c\cdot u_{\rho} + c'\cdot u_{\rho'} =  \sum_{i=1}^{d-1}c_i\cdot u_{\rho_{i}} 
\end{equation}
with $c, c' >0$. To the wall $\tau$ we associate the \emph{wall-crossing inequality}
\begin{equation}\label{ineq:wallcross}
\displaystyle  
I_{\Sigma,\tau}(h) :=  c\cdot h(u_{\rho}) + c'\cdot h(u_{\rho'}) - \sum_{i=1}^{d-1}c_i\cdot h(u_{\rho_{i}}) \geq 0,
\end{equation}
which a piecewise linear function $h \in \PL(\Sigma)$ must satisfy in order to be convex.
\end{defi}

We will often write $I_\tau(h)$ instead of $I_{\Sigma,\tau}(h)$ when there is no potential confusion in doing so. 
We are mostly interested  in cases where the fan $\Sigma$ is complete and simplicial, where
there is no choice for $\rho_1, \ldots, \rho_{d-1}$ and $\rho, \rho'$. In general, since $h$ is linear in $\sigma$ and in $\sigma'$, different choices of the $d-1$ rays $\rho_1, \ldots, \rho_{d-1}$ and the two rays $\rho, \rho'$ give rise to equivalent wall-crossing inequalities. Therefore the element $I_{\tau} \in \PL(\Sigma)^\vee$ is well-defined up to positive scaling. Notice that $I_{\tau}(h)=0$ if and only if $h$ is represented by the same linear functional at both sides on $\tau$, which happens if and only if $\tau$ is no longer a wall in the fan of lineality domains of $h$.

%

\begin{lem} \emph{(Wall-Crossing Criterion)}\label{lem:wallcross} \cite[Theorems 6.1.5--6.1.7]{toric}  Let $\Sigma$ be a full dimensional fan with convex support in $U$. A continuous piecewise linear function $h\in \operatorname{PL}(\Sigma)$ is a support function of a polyhedron $Q$ with $\Sigma_Q \preceq \Sigma$ if and only if it satisfies the wall-crossing inequality $I_{\Sigma,\tau}(h)\geq 0$, as defined in \eqref{ineq:wallcross}, for each wall $\tau$ of $\Sigma$.
\end{lem}

\begin{proof}[Sketch of Proof.] 
To check whether $h$ is convex, it suffices to check its convexity on a line segment $xy$. Furthermore, it suffices to check this condition locally,  on short segments $xy$ where $x$ and $y$ are in adjacent domains of lineality $\sigma$ and $\sigma'$. If $\tau = \sigma \cap \sigma'$ is the wall separating $\sigma$ and $\sigma'$ and $z = xy \cap \tau$, it is enough to check convexity between the extreme points $x$ and $y$ and their intermediate point $z$. One then verifies, using the linearity of $h$ in $\sigma$, that it is enough to check this when $x$ and $y$ are rays of $\sigma$ and $\sigma'$ respectively; but these are precisely the wall-crossing inequalities \eqref{ineq:wallcross}.
\end{proof}

We now describe the deformation cones for polytopes.  Note that $V$ embeds into $\operatorname{PL}(\Sigma)$ by $v\mapsto \langle v, \cdot \rangle$.  The following is a rephrasing of \cite[4.2.12, 6.3.19--22]{toric}.

\begin{prop}\label{def:nef}
Let $\Sigma$ be the normal fan of a polytope $P$.  Say $h\sim h'$ for two functions $h,h'\in \operatorname{PL}(\Sigma)$ if $h-h'$ is a globally linear function on $U$, or equivalently, if $h-h' \in V \subset \operatorname{PL}(\Sigma)$. Then:
\begin{itemize}
\item \textit{Def Cone: }$ \df(\Sigma)$ is the polyhedral cone parametrizing deformations of $P$. It is the full dimensional cone in $\operatorname{PL}(\Sigma)$ cut out by the wall-crossing inequalities $I_{\Sigma,\tau}(h)\geq 0$ for each wall $\tau$ of $\Sigma$. Its lineality space is the $d$-dimensional space $V \subset \PL(\Sigma)$ of global linear functions on $|\Sigma| = U$, corresponding to the $d$-dimensional space of translations of $P$.
\item \textit{Nef Cone: }$ \nf(\Sigma) := \operatorname{Def}(\Sigma)/V = \operatorname{Def}(\Sigma)/\sim$ is the quotient of $\df(\Sigma)$ by its lineality space $V$ of globally linear functions. It is a strongly convex cone in $\operatorname{PL}(\Sigma)/V$ parametrizing the deformations of $P$ up to translation.
\end{itemize}
\end{prop}

Two things must be kept in mind when applying Lemma \ref{lem:wallcross}. It is not true that all the wall crossing inequalities are facet defining for $\df(\Sigma)$. Furthermore, it may happen that two walls give the exact same inequality. Both situations are illustrated in \cite[Example 2.13]{deformation}.

\medskip

When $\Sigma$ is a rational fan, it has an associated toric variety $X(\Sigma)$ \cite[Chapter 6.3]{toric}, and  $\nf(\Sigma)$  is the \emph{Nef (numerically effective) cone} of the toric variety $X(\Sigma)$.
 The \emph{Mori cone} $\overline{NE}(\Sigma)$ of $\Sigma$ is
$$\overline{NE}(\Sigma) := \operatorname{Cone}\left(I_{\Sigma,\tau} \ | \ \tau\in \Sigma(d-1)\right) \subseteq \operatorname{PL}(\Sigma)^\vee.$$
The Wall-Crossing Criterion of Lemma \ref{lem:wallcross} states that 
the Nef cone and the Mori cone are dual cones in $\operatorname{PL}(\Sigma)/V$ and $(\operatorname{PL}(\Sigma)/V)^\vee$, respectively; in the toric setting, this is \cite[Theorem 6.3.22]{toric}.  The structure of the strongly convex cones $\operatorname{Nef}(\Sigma)$ and $\overline{NE}(\Sigma)$ plays an important role in the geometry of the minimal model program for associated toric varieties.  For details in this direction see \cite[\S15]{toric}.

\subsubsection{\textsf{Batyrev's criterion}\label{sec:Batyrev}}

When $\Sigma$ is simplicial, Batyrev's criterion (\cite[Lemma 6.4.9]{toric}) offers another useful test for convexity, and hence an alternative description of the deformation cone $\operatorname{Def}(\Sigma) = \operatorname{Def}(P)$ when $\Sigma = \Sigma_P$. To state it, we need the following notion.

\begin{defi} Let $\Sigma$ be a simplicial fan.  A \emph{primitive collection} $F$ is a set of rays of $\Sigma$ such that any proper subset $F'\subsetneq F$ forms a cone in $\Sigma$ but $F$ itself does not. In other words, the primitive collections of a simplicial fan correspond to the minimal non-faces of the associated simplicial complex.
\end{defi}

\begin{lem}\label{lem:Batyrev} \emph{(Batyrev's Criterion)} \cite[Theorem 6.4.9]{toric}
Let $\Sigma$ be a complete simplicial fan.
A piecewise linear function $h\in \operatorname{PL}(\Sigma)$ is in the deformation cone $\operatorname{Def}(\Sigma)$ (and hence the support function of a polytope)
if and only if
\[
\sum_{\rho\in F} h(u_\rho)
\geq
h\left(\sum_{\rho \in F} u_\rho\right) 
\]
for any primitive collection $F$ of rays of $\Sigma$.
\end{lem}


\begin{rem}
The material in this section can be rephrased in terms of triangulations of point configurations (see \cite[Section 5]{triangulations}). Deformation cones are instances of secondary cones for the collection of vectors $\{u_\rho: \rho\in \Sigma(1)\}$. The Wall-Crossing criterion Lemma \ref{lem:wallcross} is called the \emph{local folding condition} in \cite[Theorem 2.3.20]{triangulations}. The secondary cones form a \emph{secondary fan} whose faces are in bijection with the regular subdivisions of the configuration. When the configuration is \emph{acyclic} (so it can be visualized as a point configuration), this secondary fan is complete, and it is the normal fan of the \emph{secondary polytope}. Our situation is more subtle because our vector configurations is not acyclic, so the secondary fan is not complete, and there is no secondary polytope. 
\end{rem}

\section{\textsf{Reflection groups and Coxeter complexes}}\label{sec:coxeter}

In this section we review the combinatorial aspects of finite reflection groups that we will need. We refer the reader to \cite{reflection} for proofs.

\subsection{\textsf{Root systems and Coxeter complexes}}\label{subsec:rootsystems}

From now on, we will identify $V$ with its own dual by means of a positive definite inner product $\langle\cdot,\cdot\rangle:V\times V\to \mathbb{R}$. 
Any vector $v\in V$ defines a linear automorphism $s_v$ on $V$ by reflecting across the hyperplane orthogonal to $v$; that is,
\begin{equation}\label{eq:reflex}
s_v(x) := x-\dfrac{2\langle x,v \rangle}{\langle v,v \rangle}v.
\end{equation}

\begin{defi}
A \emph{root system} $\Phi$ is a finite set of vectors in an inner product real vector space $V$ satisfying: (R0) $\textrm{span}(\Phi)=V$, (R1) for each root $\alpha \in \Phi$, the only scalar multiples of $\alpha$ that are roots are $\alpha$ and $-\alpha$, and 
(R2) for each root $\alpha \in \Phi$ we have $s_\alpha(\Phi)=\Phi$.
A root system $\Phi$ is \emph{crystallographic} if it also satisfies (R3) for each pair of roots $\alpha, \beta \in \Phi$ we have that $2\langle\alpha, \beta \rangle/\langle \alpha,\alpha \rangle$ is an integer.
\end{defi}

Each root $\alpha\in \Phi$ gives rise to a hyperplane $H_\alpha=\{x\in V: \langle\alpha,x\rangle = 0\}$. This set of hyperplanes $\HH_\Phi = \{H_\alpha: \alpha\in \Phi\}$ is called the \emph{Coxeter arrangement}. The \emph{Coxeter complex} is the associated fan $\Sigma_\Phi$, which is simplicial. We will often use these two terms interchangeably, and drop the subscript $\Phi$ when the context is clear. Let $s_\alpha \in GL(V)$ be the reflection accross hyperplane $H_\alpha$; we have
\[
s_\alpha(x) = x - 2 \frac{\langle \alpha, x\rangle}{\langle \alpha, \alpha\rangle} \alpha \qquad \textrm{ for } x \in V.
\]

\begin{defi}
Let $\Phi$ be a finite root system spanning $V$ and let $W=W_\Phi$ be the subgroup of $\textrm{GL}(V)$ generated by the reflections $s_\alpha$ for $\alpha\in \Phi$. The group $W$ is a finite group, called the \emph{Weyl group} of $\Phi$.
\end{defi}

The combinatorial structure of the Coxeter complex $\Sigma_\Phi$ is closely related to the algebraic structure of the Weyl group $W_\Phi$, as we explain in the remainder of this section. Let us fix a chamber (maximal cone) of $\Sigma_\Phi$ to be the \emph{fundamental domain} $D$; recall that it is simplicial. Then the \emph{simple roots} $\Delta=\{\alpha_1,\ldots,\alpha_d\}\subset \Phi$ are the roots whose positive halfspaces minimally cut out the fundamental domain; that is,
\[
D = \{x \in V \, : \, \langle \alpha_i, x \rangle \geq 0 \textrm{ for } 1 \leq i \leq d\}.
\]
The simple roots form a basis for $V$, and we call $d = \dim V$ the \emph{rank} of the root system $\Phi$. The \emph{positive roots} are those that are non-negative combinations of simple roots; we denote this set by $\Phi^+ \subset \Phi$. We have that $\Phi = \Phi^+ \sqcup (-\Phi^+)$.

The \emph{Cartan matrix} is the $d\times d$ matrix $A$ whose entries are 
\[
A_{ij} := 2\frac{\langle \alpha_i, \alpha_j\rangle}{\langle\alpha_i, \alpha_i \rangle} \qquad \textrm{ for } 1 \leq i, j \leq d.
\]
This is a very sparse matrix: each row or column of $A$ contains at most four nonzero entries. For crystallographic root systems, the entries of the Cartan matrix are integers.

There exist positive integers $m_{ij} = m_{ji}$ such that $A_{ij}A_{ji} = 4 \cos^2(\pi/m_{ij})$. These entries form the \emph{Coxeter matrix} of $\Phi$.
This information is more economically encoded in the \emph{Dynkin diagram} $\Gamma(\Phi)$, which has vertices $\{1, \ldots, d\}$,
and an edge labelled $m_{ij}$ between $i$ and $j$ whenever $m_{ij} >2$. Labels equal to 3 are customarily omitted.

The \emph{direct sum} of two root systems $\Phi_1$ and $\Phi_2$, spanning $V_1$ and $V_2$ respectively, is the root system $\Phi_1 \oplus \Phi_2 := \{(\alpha, 0)\in V_1\oplus V_2 : \alpha \in R_1\} \cup \{(0,\beta)\in V_1\oplus V_2 : \beta\in R_2\}$ which spans $V_1 \oplus V_2$.
An \emph{irreducible} root system is a root system that is not a non-trivial direct sum of root systems. The connected components of the Dynkin diagram $\Gamma(\Phi)$ correspond to the irreducible root systems whose direct sum is $\Phi$.

\begin{theorem}\label{thm:classification} \cite[\S2]{reflection} The irreducible root systems can be completely classified into four infinite families $A_d, B_d, C_d, D_d$, the exceptional types $E_6, E_7, E_8, F_4, G_2, H_3, H_4$ in the dimensions indicated by their subscripts, and  $I_2(m)$ for $m \geq 3$. 
\end{theorem}

\begin{ex}\label{ex:roots}
The \emph{classical root systems} are
\begin{eqnarray*}
A_{d-1} & = & \{\pm(e_i - e_j) \, : \, 1 \leq i \neq j \leq d\} \\
B_d & = & \{\pm e_i \pm e_j \, : \, 1 \leq i \neq j \leq d\} \cup \{\pm e_i \, : \, 1 \leq i \leq d\}\\
C_d & = & \{\pm e_i \pm e_j \, : \, 1 \leq i \neq j \leq d\} \cup \{\pm 2e_i \, : \, 1 \leq i \leq d\}\\
D_d & = & \{\pm e_i \pm e_j \, : \, 1 \leq i \neq j \leq d\} 
\end{eqnarray*}
where $\{e_1, \ldots, e_d\}$ is the standard basis of $\RR^d$.
Notice that the root system $A_{d-1}$ spans the subspace $\RR^d_0 := \{\x \in \RR^d \, : \, x_1 + \cdots + x_d = 0\}$ of $\RR^d$. For a suitable choice of fundamental chamber, the simple roots of the classical root systems are
\begin{eqnarray*}
\Delta_{A_{d-1}} & = & \{e_1-e_2, \,  e_2-e_3, \ldots,  \,e_{d-1}-e_d\} \\
\Delta_{B_d} & = & \{e_1-e_2, \,e_2-e_3, \ldots,\, e_{d-1}-e_d, \,e_d\} \\
\Delta_{C_d} & = & \{e_1-e_2, \,e_2-e_3, \ldots, \,e_{d-1}-e_d,\, 2e_d\} \\
\Delta_{D_d} & = & \{e_1-e_2, \,e_2-e_3, \ldots, \,e_{d-1}-e_d, \,e_{d-1}+e_d\}
\end{eqnarray*}
Their Dynkin diagrams are the following. 

%
%
%
\begin{center}
\begin{tabular}{rl}
$A_d:$ & 
\begin{tikzpicture}
\draw[fill=black] 
(0,0) circle [radius=0.08]
(1,0) circle [radius=0.08]
(2,0) circle [radius=0.08]
(3,0) circle [radius=0.08]
(4,0) circle [radius=0.08]
;
\draw
(0,0) -- (1,0)
(2,0) -- (3,0) -- (4,0)
;
\draw[thick, dotted]
(1,0) -- (2,0)
;
\end{tikzpicture}
\\
$B_d, C_d:$ &  
\begin{tikzpicture}
\draw[fill=black] 
(0,0) circle [radius=0.08]
(1,0) circle [radius=0.08]
(2,0) circle [radius=0.08]
(3,0) circle [radius=0.08]
(4,0) circle [radius=0.08]
;
\draw
(0,0) -- (1,0) 
(2,0) -- (3,0) -- (4,0) node [midway,above] {4}
;
\draw[thick, dotted]
(1,0) -- (2,0)
;
\end{tikzpicture}
\\
$D_d:$ & 
\begin{tikzpicture}
\draw[fill=black] 
(0,0) circle [radius=0.08]
(1,0) circle [radius=0.08]
(2,0) circle [radius=0.08]
(3,0) circle [radius=0.08]
(4,0) circle [radius=0.08]
(5,0.5) circle [radius=0.08]
(5,-0.5) circle [radius=0.08]
;
\draw
(0,0) -- (1,0) 
(2,0) -- (3,0) -- (4,0) 
(4,0) -- (5, .5)
(4,0) -- (5, -.5)
;
\draw[thick, dotted]
(1,0) -- (2,0)
;
\end{tikzpicture}
\end{tabular}
\end{center}
The root system $C_3$, illustrated in Figure \ref{fig:C3}, will serve as our running example in this section.
\end{ex}

\begin{figure}[h] 
\centering
\begin{tikzpicture}
    \draw (0, 0) node[inner sep=0]
{\includegraphics[height=7cm]{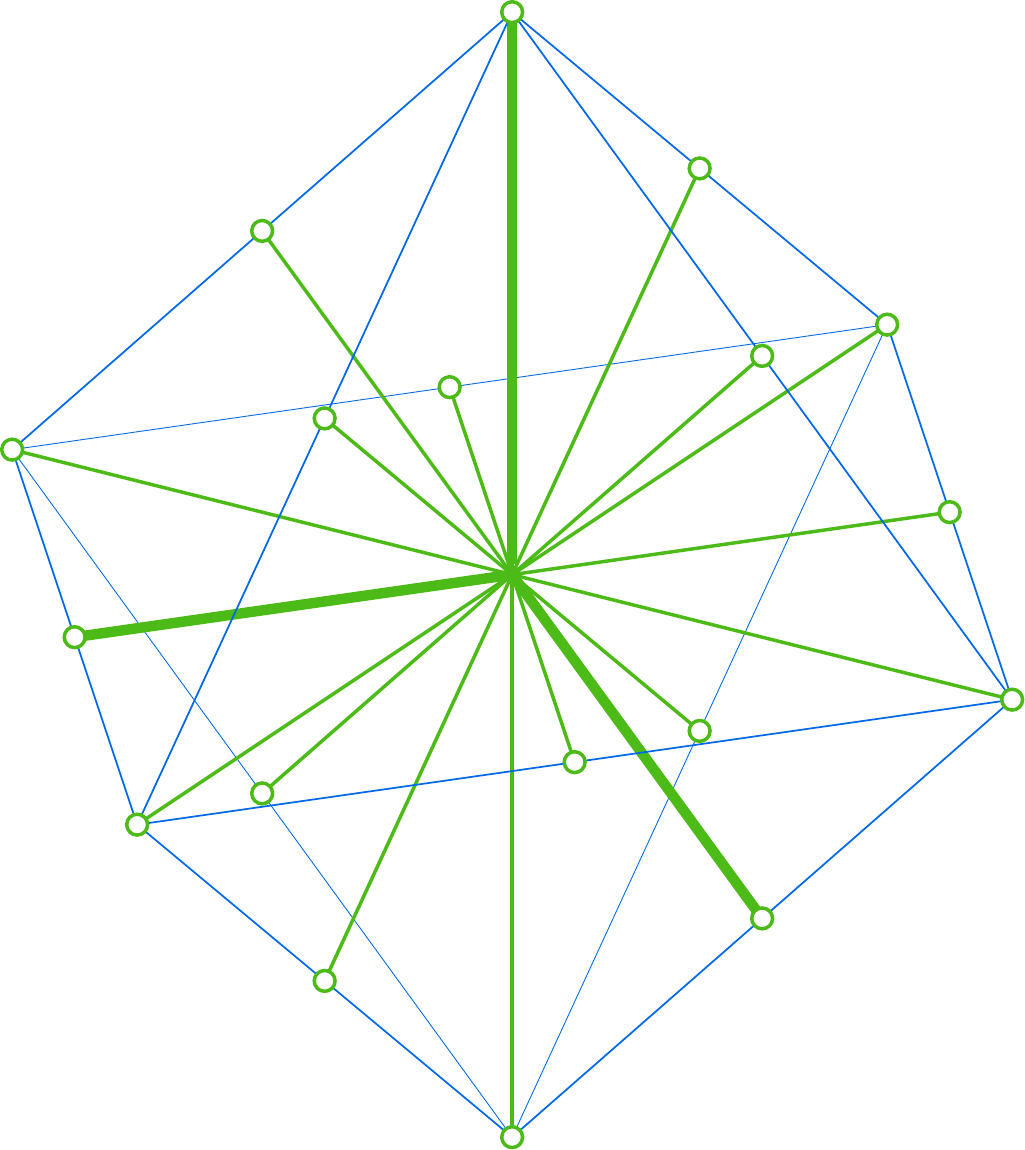}
\includegraphics[height=2cm]{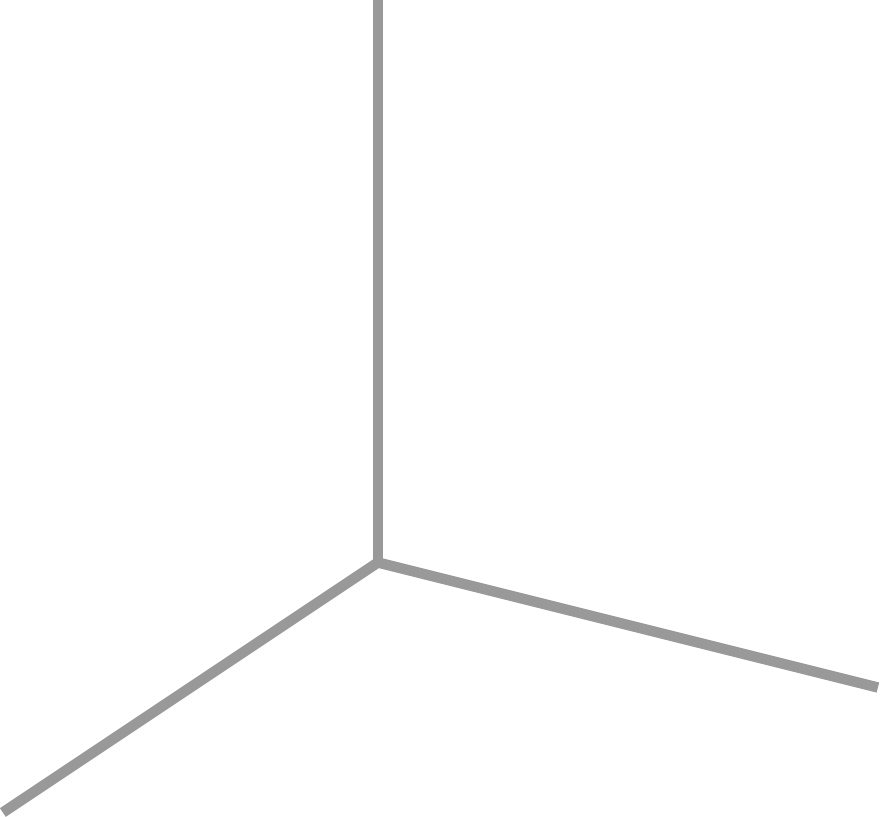} \qquad 
\includegraphics[height=7cm]{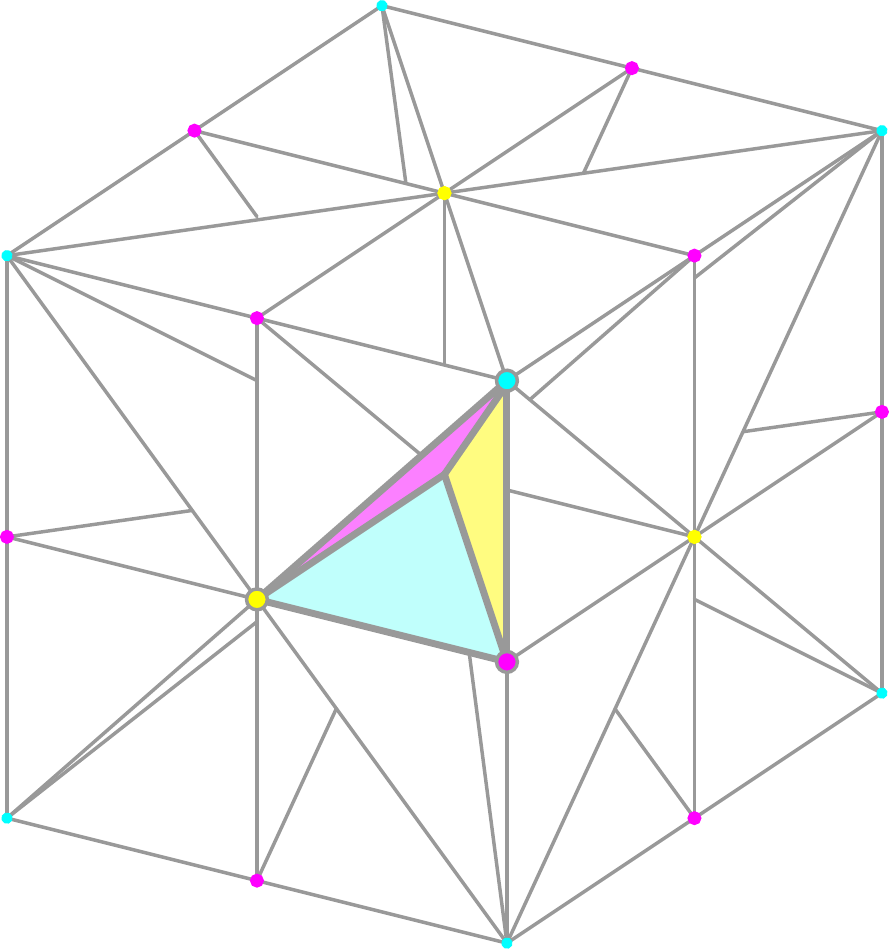}};  \qquad ;
    \draw (-7.9, -.6) node {$\alpha_1$};
    \draw (-3.2, -2.4) node {$\alpha_2$};
    \draw (-5.3, 3.4) node {$\alpha_3$};
    \draw (2.855, -1) node {$\myomega_1$};
    \draw (5.4, -1.5) node {$\myomega_2$};
    \draw (5.3, 1.1) node {$\myomega_3$};
    \draw (5.4, -0.4) node {$s_1$};
    \draw (4.1, 0.1) node {$s_2$};
    \draw (4.2, -1.3) node {$s_3$};
\end{tikzpicture}

\caption{(a) The root system $C_3$ consists of $24$ roots, which are the vertices and edge midpoints of a regular octahedron. The simple roots $\alpha_1, \alpha_2, \alpha_3$ are emphasized. (b) The Coxeter complex of $C_3$ has 48 chambers. One of them, the fundamental domain, is emphasized; its rays contain the fundamental weights $\myomega_1, \myomega_2, \myomega_3$, and its walls determine the simple reflections $s_1, s_2, s_3$. \label{fig:C3}
}

\end{figure}
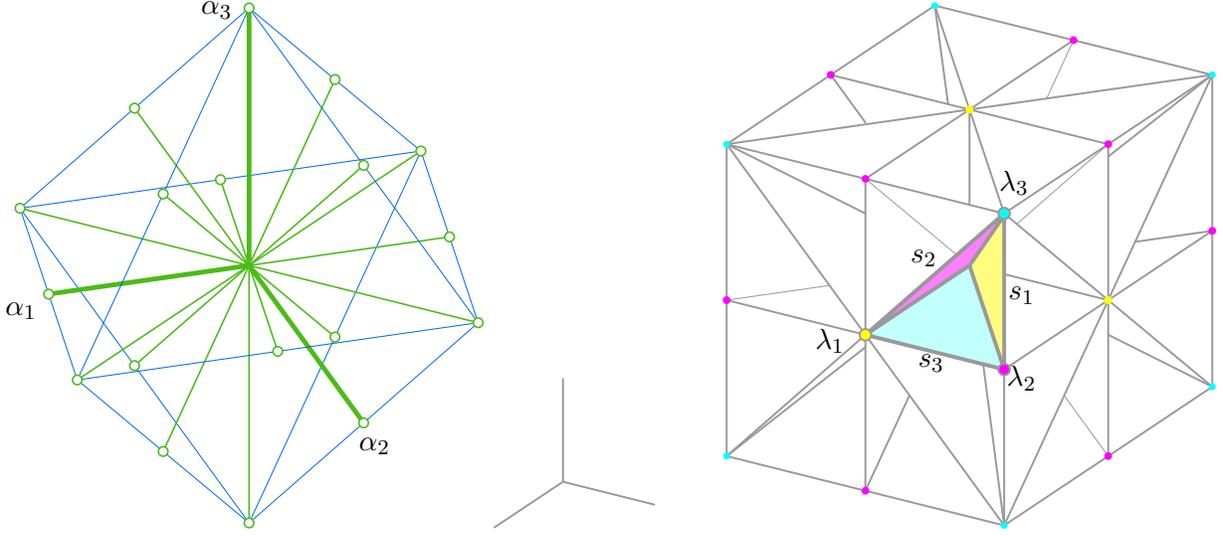

\begin{defi}
If $\Phi$ is a root system on $V$, the \emph{coroot} $\alpha^\vee$ of a root $\alpha$ is an element of $V^*$ which is defined as
\[
\alpha^\vee = \frac{2}{\langle \alpha, \alpha \rangle} \alpha.
\]
after identifying $V^*\simeq V$ by the inner product $\langle \cdot, \cdot \rangle$.  The coroots form the \emph{dual root system} $\Phi^\vee$.
\end{defi}

Notice that the reflection across $\alpha_i$ can be rewritten simply as
\begin{equation}\label{eq:s_i}
s_i(x) = x -  \langle x, \alpha_i^\vee \rangle \alpha_i
\end{equation}
Also notice that the Cartan matrix can be rewritten as
\begin{equation}\label{eq:A}
A_{ij} =  \langle \alpha^\vee_i , \alpha_j \rangle.
\end{equation}
This implies that the Cartan matrix of the dual root system $\Phi^\vee$ is $A^T$.  

\begin{defi}
Let the \emph{fundamental weights} $(\myomega_1,\ldots,\myomega_d)$ form the basis of $V$ dual to the simple coroots $(\alpha_1^\vee,\ldots,\alpha_d^\vee)$; that is, $\langle \myomega_i,\alpha^\vee_j \rangle = \delta_{ij}$. Let the \emph{fundamental weight conjugates} or $\emph{rays}$ of $\Phi$ be 
\[
\R = \R_\Phi :=  W\{\myomega_1,\ldots,\myomega_d\}.
\]
Each element of $\R_\Phi$ can be expressed as $w \myomega_i$ for a unique $i$; the choice of $w$ is not unique. The rays of the Coxeter arrangement $\mathcal H_\Phi$ are exactly the positive spans of the vectors in $\R_\Phi$, explaining our terminology. 

Similarly, let the \emph{fundamental coweights} $(\myomega_1^\vee,\ldots,\myomega_d^\vee)$ form the basis of $V$ dual to the simple roots $(\alpha_1, \ldots, \alpha_d)$. Clearly $\myomega_i^\vee = \frac12 \langle \alpha_i, \alpha_i \rangle \myomega_i$. 
\end{defi}

Let $e_1, \ldots, e_d$ be an orthonormal basis for $\RR^d$. Let $e_S := \sum_{i\in S} e_i$ for $S \subseteq [d]$ and denote $\1 = e_{[d]} =  (1, \ldots, 1) \in \RR^d$. 
For $\x \in \RR^d$ let $\overline{\x}$ denote its coset representative in $\RR^d/\RR\1$.

\begin{ex}\label{ex:weights} The fundamental weights of the classical root systems are:
\begin{eqnarray*}
A_{d-1}:  &  & \{\e_1, \e_1+\e_2, \ldots, \e_1+\cdots+\e_{d-1}\} \\
B_d: &  & \{e_1, e_1+e_2, \ldots, e_1+\cdots+e_{d-1}, (e_1 + \cdots + e_d)/2\} \\
C_d: &  & \{e_1, e_1+e_2, \ldots, e_1+\cdots+e_{d-1}, 
e_1 + \cdots + e_d\} \\
D_d: &  & \{e_1, e_1+e_2, \ldots, e_1+\cdots+e_{d-2}, 
(e_1 + \cdots + e_{d-1} -  e_d)/2, 
(e_1 + \cdots + e_{d-1} + e_d)/2\}
\end{eqnarray*}
\end{ex}

In light of \eqref{eq:A}, the transition matrix between the roots and the fundamental weights is the transpose of the Cartan matrix:
\begin{equation}\label{alphatolambda}
\alpha_j = \sum_{i=1}^d A_{ij} \myomega_i \qquad \textrm{and} \qquad 
\myomega_j = \sum_{i=1}^d A^{-1}_{ij} \alpha_i \qquad \qquad \textrm{ for } 1 \leq j \leq d,
\end{equation}
and we have
\begin{equation} \label{Ainverse}
\langle \myomega^\vee_i, \myomega_j \rangle = A^{-1}_{ij}.
\end{equation}

We say $\Phi$ is \emph{simply laced} if it is of type $ADE$; that is, its Dynkin diagram has no labels greater than $3$. These root systems are self dual; there is no distinction between roots and coroots, or between weights and coweights.

\subsection{\textsf{Weyl groups, parabolic subgroups, and the geometry of the Coxeter complex}}\label{subsec:groups}

Let $\Phi$ be a finite root system spanning $V$ and let $W=W_\Phi$ be its Weyl group; recall that it is finite.
Let $\Delta = \{\alpha_1, \ldots, \alpha_d\}$ be a choice of simple roots of $\Phi$, and let $s_i = s_{\alpha_i}$ be the reflection across the hyperplane $H_{\alpha_i}$ orthogonal to $\alpha_i$ for $1 \leq i \leq d$.

\begin{prop}
The Weyl group $W$ of the root system $\Phi$ is generated by the set of simple reflections $S := \{s_1, \ldots, s_d\}$, with presentation given by the Coxeter matrix as follows:
\begin{equation}\label{eq:presentation}
W = \langle s_1, \ldots, s_d \ | \ (s_is_j)^{m_{ij}} = e \ \textrm{ for } 1\leq i,j\leq d)\rangle
\end{equation}
\end{prop}

\begin{ex}\label{ex:Weyl}
The Weyl groups of the classical root systems are:
\begin{eqnarray*}
W_{A_{d-1}} &=& \{\textrm{permutations of } [d]\} \\
W_{B_d} = W_{C_d} &=& \{\textrm{signed permutations of } [d]\} \\
W_{D_d} &=& \{\textrm{evenly signed permutations of } [d]\}. 
\end{eqnarray*}
As matrix groups, $W_{A_{d-1}}$ is the set of $d \times d$ permutation matrices, 
$W_{B_d}=W_{C_d}$ is the set of $d \times d$ ``generalized permutation matrices" whose non-zero entries are $1$ or $-1$, and $W_{D_d}$ is the subgroup of $W_{B_d}$ whose matrices involve an even number of $-1$s.
\end{ex}

The action of $W$ on $V$ induces an action on the Coxeter complex $\Sigma_\Phi$. Every face of $\Sigma_\Phi$ is $W$-conjugate to a unique face of the fundamental domain.
This action behaves especially well on the top-dimensional faces:

\begin{prop}
The Weyl group $W$ acts regularly on the set $\Sigma_\Phi(d)$ of chambers of the Coxeter arrangement; that is, for any two chambers $\sigma$ and $\sigma'$ there is a unique element $w \in W$ such that $w \cdot \sigma = \sigma'$. In particular, the chambers of the Coxeter arrangement are in bijection with $W$.
\end{prop}

The previous proposition implies that a different choice $wD$ of a fundamental domain  (where $w\in W$) gives rise to a new set of simple roots $w \Delta$ that is linearly isomorphic to the original set $\Delta$ of simple roots, since $W$ acts by isometries. It follows that the presentation for the Weyl group in \eqref{eq:presentation} and the Cartan matrix $A$ are independent of the choice of fundamental domain $D$.

\medskip

The lower dimensional faces of $\Sigma_\Phi$ correspond to certain subgroups of $W$ and their cosets. The \emph{parabolic subgroups} of $W$ are the subgroups
\[
W_I := \langle s_\alpha: \alpha\in I\rangle \subseteq W \qquad \textrm{ for each } I \subseteq \Delta.
\]
They are in bijection with the faces of the fundamental domain, where $W_I$ is mapped to the face
\[
C_I :=  \{x\in D: \langle x,\alpha \rangle = 0 \textrm{ for all }\alpha\in I, \langle x,\alpha \rangle \geq 0 \textrm{ for all }\alpha\in\Delta\backslash I\}  
\]
The \emph{parabolic cosets} are the cosets of parabolic subgroups.

\begin{prop}\label{prop:stab}
The faces of the Coxeter complex are in bijection with the parabolic cosets of $W$, where the face $F$ is labeled with the parabolic coset 
$\{w \, : \, F \subseteq wD\}$.  
More explicitly, the face $C_I$ of the fundamental domain is labeled with the parabolic subgroup $W_I$, and its $W$-conjugate $vC_I$ is labeled with the coset $vW_I$ for $v \in W$.
\end{prop}

Two special cases, stated in the following corollaries, are especially important to us.

\begin{cor}
The walls of the Coxeter complex are labeled by the pairs $\{w, ws_i\} = w W_{\{i\}}$ for $w \in W$ and $s_i \in S$. The wall labeled $\{w, ws_i\}$ separates the chambers labeled $w$ and $ws_i$. This correspondence is bijective, up to the observation that $w W_{\{i\}} = ws_i W_{\{i\}}$.
\end{cor}

\begin{cor}
The $d$ rays of the fundamental domain are spanned by the fundamental weights $\{\myomega_1, \ldots, \myomega_d\}$, and the rays of the Coxeter complex are spanned by the fundamental weight conjugates $\R = W\{\myomega_1, \ldots, \myomega_d\}$. These correspondences are bijective.
\end{cor}

Note that the faces of the fundamental chamber are given by
\begin{eqnarray*}
C_I &=& D \medcap \Big( \, \bigcap_{i\in I} H_{\alpha_i}\Big) \\
&=& \operatorname{cone}(\myomega_i \ | \ i\notin I) \qquad \textrm{ for } I\subseteq \Delta.
\end{eqnarray*}
The parabolic subgroups arise as isotropy groups for the action of $W$ on $V$ \cite[Theorem 1.12, Proposition 1.15]{reflection}:
\begin{theorem}\label{thm:isotropy}
The isotropy group of the face $C_I$ is precisely the parabolic subgroup $W_I$. More generally, if $V'$ is any subset of $V$ then the subgroup of $W$ fixing $V'$ pointwise is generated by those reflections $s_\alpha$ whose normal hyperplane $H_\alpha$ contains $V'$.
\end{theorem}

Let $[\pm d] = \{1, 2, \ldots, d, -1, -2, \ldots, -d\}$. Say that a subset $S$ of $[\pm d]$ is \emph{admissible} if it is nonempty and $j \in S$ implies that $-j \notin S$. In this case, write $S  \subseteq [\pm d]$, and let $e_S = e_A - e_B$ where $A = \{a \in [d] \, : \, a \in S\}$ and $B = \{b \in [d] \, : \, -b \in S\}$.


\begin{ex}\label{ex:Wweights}
For the classical root systems, the rays or fundamental weight conjugates are:
\begin{eqnarray*}
\R_{A_{d-1}} &=& \{\e_S \, : \, \emptyset \subsetneq S \subsetneq [d]\} \\
\R_{B_{d}} &=& \{e_S \, : \, \textrm{admissible } S \sqsubseteq [\pm d] \, , \, |S| \leq d-1 \} \cup \Big\{\frac12 e_S \, : \, \textrm{admissible } S \sqsubseteq [\pm d] \, , \, |S| = d \Big\} \\
\R_{C_{d}} &=& \{e_S \, : \, \textrm{admissible } S \sqsubseteq [\pm d]\} \\
\R_{D_{d}} &=& \{e_S \, : \, \textrm{admissible } S \sqsubseteq [\pm d] \, , \, |S| \leq d-2 \} \cup \Big\{\frac12 e_S \, : \, \textrm{admissible } S \sqsubseteq [\pm d] \, , \, |S| = d \Big\}. 
\end{eqnarray*}
\end{ex}

\section{\textsf{Coxeter permutohedra and some important deformations}} \label{sec:examples}

One of the main goals of this paper is to describe the cone of deformations of the $\Phi$-permutohedron; we will do so in Theorem \ref{thm:main}. Before we do that, we motivate that result by discussing some notable examples of generalized $\Phi$-permutohedra in this section. 

Throughout this section, let $\Phi$ be a root system and $W$ be its Weyl group. 
The following definitions will play an important role.

\begin{defi}\label{def:Bruhat}
Define the \emph{length} $l(w)$ of an element $w \in W$ to be the smallest $k$ for which there exists a factorization $w=s_{i_1} \cdots s_{i_k}$ into simple reflections $s_{i_1}, \ldots, s_{i_k} \in S$.

\noindent $\bullet$ The \emph{Bruhat order} on $W$ is the poset defined by decreeing that $w < w s_\alpha$ for every element $w \in W$ and reflection $s_\alpha$ with $\alpha \in \Phi$ such that $l(w) < l(w s_\alpha)$.

\noindent $\bullet$ The \emph{weak order} on $W$ is the poset defined by decreeing that $w < ws_i$ for every element $w \in W$ and simple reflection $s_i$ with $\alpha_i \in \Delta$ such that $l(w) < l(ws_i)$.
\end{defi}

\subsection{\textsf{The Coxeter permutohedron}}

\begin{defi}
The \emph{standard Coxeter permutohedron of type $\Phi$} or \emph{$\Phi$-permutohedron} is the Minkowski sum of the positive roots of $\Phi$; that is, the zonotope of the Coxeter arrangement $\HH_\phi$:
\begin{eqnarray*}
\Pi_\Phi &:=&  \sum_{\alpha \in \Phi^+} [-\alpha/2,\alpha/2] \\
&=& \operatorname{conv}\{w \cdot \rho : w\in W\},
\end{eqnarray*}
where $\rho = \frac12(\sum_{\alpha \in \Phi^+} \alpha) = \myomega_1 + \cdots + \myomega_d$ is the sum of the fundamental weights.
\end{defi}

The 1-skeleton of the $\Phi$-permutohedron can be identified with the Hasse diagram of the weak order on $W$:
vertices $w  \rho$ and $w'  \rho$ are connected by an edge if and only if $w' = ws_i$ for some simple reflection $s_i$, and in that situation $w < w'$ in the weak order  if and only if $\langle w  \rho ,\rho \rangle > \langle w'  \rho, \rho \rangle$.

\begin{defi}
A \emph{generalized Coxeter permutohedron} or \emph{Coxeter polymatroid} is a deformation of the $\Phi$-permutohedron $\Pi_\Phi$; that is, a polytope whose normal fan coarsens the Coxeter complex $\Sigma_\Phi$.
\end{defi}

We collect the results of this section in the following proposition. The following subsections include precise definitions and further details.

\begin{prop}\label{prop:examples}
The following families of polytopes are deformations of Coxeter permutohedra:

1. the \emph{weight polytopes} describing the representations of semisimple Lie algebras \cite{FultonHarris},

2. the \emph{Coxeter graphic zonotopes} of Zaslavsky \cite{Zaslavskysigned},

3. the \emph{Coxeter matroids} of Gelfand--Serganova \cite{Coxetermatroids, GelfandSerganova},

4. the \emph{Coxeter root cones} of Reiner \cite{Reiner} and Stembridge \cite{Stembridge}, and

5. the \emph{Coxeter associahedra} of Hohlweg-Lange-Thomas \cite{HohlwegLangeThomas}.

\end{prop}

\subsection{\textsf{Weight polytopes}}

\begin{defi}
The \emph{weight polytope} $P_\Phi(x)$ of a point $x \in V$ is the convex hull of the orbits of $x$ under the action of the Weyl group $W$:
\[
P_\Phi(x) := \operatorname{conv}\{w\cdot x : w\in W\}.
\]
\end{defi}

These polytopes are of fundamental importance in the theory of Lie algebras\footnote{when $\Phi$ is crystallographic}. \cite{FultonHarris, HumphreysLie, Khare} A semisimple complex Lie algebra $\mathfrak{g}$ has an associated root system $\Phi$ which controls its representation theory. The irreducible representations $L(\myomega)$ of $\mathfrak{g}$ are in bijection with the  points $\myomega \in D \cap \Lambda$, where $D$ is the dominant chamber of the root system $\Phi$ and $\Lambda$ is the \emph{weight lattice} generated by the fundamental weights. The representation $L(\myomega)$ decomposes as a direct sum of weight spaces $L(\myomega)_\mu$ which are indexed precisely by the lattice points $\mu$ in the weight polytope $P_\Phi(x)$.

\begin{proof}[Proof of Proposition \ref{prop:examples}.1]
Every edge of $P_\Phi(x)$ is parallel to a root in $\Phi$ by \cite[Lemma 4.13]{LiCaoLi}, so Proposition \ref{prop:zonoedges} implies that weight polytopes are generalized $\Phi$-permutohedra.
\end{proof}

\begin{rem}
An important special case of this construction is the \emph{root polytope} of $\Phi$, which is the convex hull of the roots.
\end{rem}


\subsection{\textsf{Coxeter graphic zonotopes}}

\begin{defi}
For any subset $\Psi \subseteq \Phi^+$ of positive roots, we define the \emph{Coxeter graphic zonotope} to be the Minkowski sum
\[
Z(\Psi) = \sum_{\alpha \in \Psi} [0, \alpha].
\]
\end{defi}

In type $A_{n-1}$, a subset $\Psi$ of $\Phi^+ = \{e_i - e_j \, : \, 1 \leq i < j \leq n\}$ corresponds to a graph $G_\Psi$ with vertex set $[n]$ and an edge connecting $i$ and $j$ whenever $e_i-e_j \in \Psi$. The definition above is the usual definition of the graphic zonotope of $G_\Psi$.

\begin{proof}[Proof of Proposition \ref{prop:examples}.2]
The normal fan of $Z(\Psi)$ is given by the subarrangement $\HH_\Psi \subseteq \HH_\Phi$ consisting of the normal hyperplanes to the roots in $\Psi$. This is clearly a coarsening of $\Sigma_\Phi$, so Coxeter graphic zonotopes are indeed generalized $\Phi$-permutohedra.
\end{proof}

\subsection{\textsf{Coxeter matroids}}

Gelfand and Serganova \cite{GelfandSerganova} introduced Coxeter matroids, a generalization of matroids that arises in the geometry of homogeneous spaces $G/P$. The book \cite{Coxetermatroids} offers a detailed acount; here we give a brief sketch. Throughout this subsection, fix a parabolic subgroup $W_I$ of the Weyl group $W$ generated by $I$.


%
%
%
%


Let $\myomega_I = \sum_{i \notin I} \myomega_i$. 
As we will see in Proposition \ref{prop:orbit}, the quotient $W/W_I$ is in bijection with the set of vertices of the weight polytope $Q(W/W_I) := P_\Phi(\myomega_I)$. The coset $\overline w \in W/W_I$ corresponds to the vertex $\delta_I(\overline w) = w  \myomega_I$, which is independent of the choice of $w \in \overline w$ because $\myomega_I \in C_I$.

\begin{defi} \label{thm:Coxetermatroids} 
For a subset $M\subseteq W/W_I$ define the polytope
\begin{equation} \label{eq:Q(M)}
Q(M) := \operatorname{conv}\{ \delta_I(\overline w) \ | \ \overline w \in M\} \subseteq Q(W/W_I).
\end{equation}
Then $M$ is a \emph{Coxeter matroid} if every edge of $Q(M)$ is parallel to a root in $\Phi$.
\end{defi}

Coxeter matroids are originally defined differently in terms of a Coxeter analog of the greedy algorithm, and the equivalence \cite[Theorem 6.3.1]{Coxetermatroids} to the definition given above is a generalization of the GGMS theorem in \cite{GGMS}.  If $M$ is a Coxeter matroid, we call $Q(M)$ its \emph{base polytope} or \emph{Coxeter matroid polytope}. 

\begin{proof}[Proof of Proposition \ref{prop:examples}.3]
Theorem \ref{thm:Coxetermatroids} and Proposition \ref{prop:zonoedges} readily imply that Coxeter matroid polytopes are generalized $\Phi$-permutohedra.
\end{proof}

In type $A_{n-1}$, when $W=S_n$ and $W_I= \langle s_1, \ldots, s_{k-1}, s_{k+1}, \ldots, s_{n-1} \rangle$
 is a maximal parabolic subgroup, the quotient $W/W_I$ is in bijection with the collection of $k$-subsets of $[n]$, and a $(W,W_I)$-matroid is precisely a matroid on $[n]$ of rank $k$. 
 
It is worth mentioning an especially interesting family of Coxeter matroid polytopes. 
Given $u \leq v$ where $v$ is a minimal-length coset representative of $W/W_I$, the 
 \emph{Bruhat interval polytope} is defined to be
$ Q^I_{u,v} := \conv \{ z \cdot \lambda_J \, : \, u \leq z \leq v \textrm{ in the Bruhat order}\}.$
These are the Coxeter matroids corresponding to Richardson varieties; see \cite{CDM19, TW15}.

\subsection{\textsf{Coxeter root cones}}

\begin{defi}
For any subset $\Psi \subseteq \Phi$ of roots we define the \emph{Coxeter root cone}
\[
\cone(\Psi) = \left\{ \sum_{\alpha \in \Psi} c_\alpha \alpha \, : \, c_\alpha \geq 0 \textrm{ for all } \alpha \in \Psi\right\}
\]
\end{defi}

Coxeter root cones are dual to the \emph{Coxeter cones} of Stembridge \cite{Stembridge}. Furthermore, pointed Coxeter root cones are in one-to-one correspondence with Reiner's \emph{parsets} \cite{Reiner}. In type $A$, these families are in bijection with \emph{preposets} and \emph{posets} on $[n]$, respectively. This correspondence sends $\cone(\Psi)$ to the (pre)poset given by $i < j$ if $e_i-e_j \in \cone(\Psi)$.

\begin{proof}[Proof of Proposition \ref{prop:examples}.5]
Every face of $\cone(\Psi)$ is generated by roots, so its dual face in the normal fan $\Sigma_{\cone(\Psi)}$ is cut out by hyperplanes in the Coxeter arrangement. Therefore any Coxeter root cone is an extended Coxeter generalized permutohedron.
\end{proof}

\subsection{\textsf{Coxeter associahedra}}

A \emph{Coxeter element} $c$ of $W$ is the product of the simple reflections of $W$ in some order. Reading \cite{Reading} introduced the \emph{Cambrian fan} $\mathcal{F}_c$, a complete fan with rich combinatorics and close connections with the theory of cluster algebras \cite{ReadingSpeyer}. Hohlweg, Lange, and Thomas constructed the \emph{Coxeter associahedron} $\textsf{Asso}_c(W)$, a polytope whose normal fan is the Cambrian fan $\mathcal{F}_c$; for details, see \cite{HohlwegLangeThomas}.

In type $A$, one choice of Coxeter element gives rise to Loday's realization of the  \emph{associahedron}, a polytope with $C_n = \frac{1}{n+1} {2n \choose n}$ vertices discovered by Stasheff in homotopy theory.  \cite{Loday} In type $B$, one choice  gives Bott and Taubes's \emph{cyclohedron}, which originally arose in knot theory. \cite{BottTaubes}

\begin{proof}[Proof of Proposition \ref{prop:examples}.5]
This holds since
 the Cambrian fan $\mathcal{F}_c$ coarsens the Coxeter fan. \cite{ReadingSpeyer}
\end{proof}

\section{\textsf{Deformations of Coxeter permutohedra: the $\Phi$-submodular cone}} \label{sec:Phisubmodularcone}

Our next goal is to describe the deformation cone of a Coxeter permutohedron.  Throughout this section, we let $\Phi$ be a fixed finite root system of dimension $d$.  Let $W$ be the corresponding Weyl group, $\Sigma = \Sigma_\Phi$ the Coxeter complex, $D$ a fixed choice of a fundamental chamber, $A$ the Cartan matrix, and $\R = W \{\myomega_1, \ldots, \myomega_d\}$ the set of conjugates of the fundamental weights $\{\myomega_1, \ldots, \myomega_d\}$.

Recall that a piecewise linear function on a fan is uniquely determined by its restriction to the rays of the fan. Since each ray of the Coxeter complex $\Sigma_\Phi$ contains a conjugate to a fundamental weight, and this correspondence is bijective, we may identify the space $\operatorname{PL}(\Sigma_\Phi)$ of piecewise-linear functions on $\Sigma_\Phi$, with the space $\RR^\R$ of functions from $\R$ to $\RR$.

\subsection{\textsf{$\Phi$-submodular functions}}

\begin{defi} 
A function $h: \R \rightarrow \RR$ is \emph{$\Phi$-submodular} if the following equivalent conditions hold:

 $\bullet$ 
$h$ is in the deformation cone $\operatorname{Def}(\Sigma_\Phi)$ of the Coxeter complex of $\Phi$.

 $\bullet$ 
When regarded as a piecewise linear function in $\PL(\Sigma_\Phi)$, the function $h$  is convex.

 $\bullet$ 
$h$ is the support function of the polytope $P_h:= \{v \in V \, : \, \langle \myomega, v \rangle \leq h(\myomega) \textrm{ for all } \myomega \in \R\}$ that is a generalized $\Phi$-permutohedron.


\end{defi}

The correspondence between $\Phi$-submodular functions $h$ and generalized $\Phi$-permutohedra $P_h$ is a bijection by Theorem \ref{th:defcone}. Furthermore, every defining inequality $\langle \myomega, v \rangle \leq h(\myomega)$ of $P_h$ is \emph{tight}, in the sense that $\displaystyle h(\myomega) = \max_{v \in P_h} \langle \myomega, v \rangle$ for all $\myomega \in \R$.
We now describe $\operatorname{Def}(\Sigma_\Phi)$, the \emph{$\Phi$-submodular cone}.


\begin{theorem} \label{thm:main}
A function $h: \R \rightarrow \RR$ is \emph{$\Phi$-submodular} if and only if the following two equivalent sets of inequalities hold:

\begin{enumerate}
\item 
\emph{(Local $\Phi$-submodularity)}
For every element $w \in W$ of the Weyl group and every simple reflection $s_i$ and corresponding fundamental weight $\myomega_i$,
\begin{equation}\label{ineq:localsubmod}
h(w  \myomega_i) +  h(ws_i  \myomega_i) \geq 
  \sum_{j \in N(i)}- 
  A_{ji} \,
  h(w  \myomega_j) 
\end{equation}
where $A$ is the Cartan matrix and $N(i)$ is the set of neighbors of $i$ in the Dynkin diagram.
\item
\emph{(Global $\Phi$-submodularity)}
For any two conjugates of fundamental weights $\myomega, \myomega' \in \R$
\begin{equation}\label{ineq:Batyrev}
h(\myomega) + h(\myomega') \geq
h(\myomega + \myomega') 
\end{equation}
where $h$ is regarded as a piecewise-linear function on $\Sigma_\Phi$.
\end{enumerate}
 \end{theorem}

\begin{rem}
By the sparseness of the Cartan matrix, the local $\Phi$-submodular inequalities \eqref{ineq:localsubmod} have at most three terms on the right hand side, given by the neighbors of $i$ in the Dynkin diagram.
\end{rem}

\begin{rem}
To interpret the global $\Phi$-submodular inequalities \eqref{ineq:Batyrev} directly in terms of the function $h \in \RR^\R$, we need to  find the minimal cone $C$ of $\Sigma_\Phi$ containing $\myomega + \myomega'$. If $\R_C = C \cap \R$ is the set of conjugates of fundamental weights in the cone $C$, we can write $\myomega + \myomega' = \sum_{w \in \R_C} c_w w$ for a unique choice of non-negative constants $c_w$, and \eqref{ineq:Batyrev} means that $h(\myomega) + h(\myomega')  \geq  \sum_{w \in \R_C} c_w h(w)$. In particular, \eqref{ineq:Batyrev} holds trivially when $\myomega$ and $\myomega'$ span a face of $\Sigma_\Phi$.
\end{rem}

\begin{proof}[Proof of Theorem \ref{thm:main}.1]
We know that the deformation cone $\df(\Sigma_\Phi)$ is given by the wall crossing inequalities of Lemma \ref{lem:wallcross}. We first compute them for the walls of the fundamental domain $D$.

Let us apply Definition \ref{defi:wallcross} to the wall $H_i = H_{\alpha_i}$ of $D$ orthogonal to the simple root $\alpha_i$, which separates the chambers $D$ and $s_i  D$. Notice that the only ray of $D$ that is not on the wall $H_i$ is precisely the one spanned by the fundamental weight $\myomega_i$. Similarly, the only ray of $s_iD$ that is not on $H_i$ is the one spanned by  $s_i  \myomega_i \in \R$. Therefore we need to find the coefficients such that
\[
c \myomega_i + c' s_i  \myomega_i = \sum_{i\neq j} c_j\myomega_j. 
\]
Since $\myomega_i$ and $s_i \myomega_i$ are symmetric across the wall $H_i$, the coefficients $c$ and $c'$ in the equation above are equal, and we may set them both equal to $1$. To compute the coefficient $c_j$ for $j \neq i$, let us take the inner product of both sides with $\alpha^\vee_j$. We obtain that
\[
\langle s_i  \myomega_i, \alpha^\vee_j \rangle = c_j, 
\]
keeping in mind that the bases $\{\alpha^\vee_1, \ldots, \alpha^\vee_d\}$ and $\{\myomega_1, \ldots, \myomega_d\}$ are dual. Thus
\[
c_j =
 \left\langle
\myomega_i - {\langle \myomega_i, \alpha^\vee_i \rangle}\alpha_i \,  , \, \alpha^\vee_j 
\right\rangle = 
0 - \langle \alpha_i,\alpha^\vee_j \rangle = 
-A_{ji}.
\]
It follows that 
\begin{equation} \label{eq:wall}
\myomega_i + s_i  \myomega_i = \sum_{i\neq j} - A_{ji} \myomega_j,
\end{equation}
so the wall-crossing inequality is 
\begin{equation}\label{ineq:wallcox}
h(\myomega_i) + h(s_i  \myomega_i)
\geq
\sum_{j\neq i} -A_{ji} h(\myomega_j).
\end{equation}
It remains to observe that $A_{ji} = 0$ unless $i$ and $j$ are neighbors in the Dynkin diagram.

More generally, consider the wall-crossing inequality for the wall $wH_i$, which separates chambers $w  D$ and $ws_i  D$. The rays of these chambers that are not on the wall are $w \myomega_i$ and $w s_i   \myomega_i$, and 
\[
w \myomega_i + w s_i   \myomega_i = \sum_{j \in N(i)} - A_{ji} \,\, w  \myomega_j. 
\]
by \eqref{eq:wall}.
Therefore the wall-crossing inequalities are indeed the ones given in \eqref{ineq:localsubmod}.
\end{proof}

%
%

\begin{proof}[Proof of Theorem \ref{thm:main}.2]
Since the Coxeter complex is simplicial, the deformation cone $\df(\Sigma_\Phi)$ is also given by Batyrev's condition as described in Lemma \ref{lem:Batyrev}. To apply it, we need to understand the primitive collections of rays in $\Sigma_\Phi$.

The Coxeter complex $\Sigma_\Phi$ is \emph{flag}, in the sense that a set of rays $R_1, \ldots, R_k$ forms a $k$-face of $\Sigma$ if and only if every pair of them forms a $2$-face of $\Sigma$. \cite[p. 29]{AbramenkoBrown} This is equivalent to saying that the primitive collections are the pairs that do not form a $2$-face. The desired result follows.
\end{proof}

\begin{defi}
A \emph{discrete $\Phi$-submodular function} is a $\Phi$-submodular function $h: \R_{\Phi} \to \ZZ$ whose values are integers.  
\end{defi}

In type $A_{n-1}$, discrete submodular functions $h: \R_{A_{n-1}} \to \ZZ$ are in bijection with lattice generalized permutohedra. This fact generalizes as follows.

\begin{prop}
Let $\Phi$ be a crystallographic root system
and $\Lambda_{R^\vee}$ be the lattice generated by the coroots.
The map $h \mapsto P_h$ is a bijection between the discrete $\Phi$-submodular functions $h: \R_{\Phi} \to \ZZ$ and the generalized $\Phi$-permutohedra $P_h$ which are lattice polytopes with respect to the coroot lattice $\Lambda_{R^\vee}$.
\end{prop}

\begin{proof}
The weight lattice $\Lambda_W \subset V$ generated by the fundamental weights and the coroot lattice $\Lambda_{R^\vee}$ generated by the simple coroots are dual lattices under the inner product $\langle \cdot, \cdot \rangle$ on $V$.

%

The Coxeter fan $\Sigma_\Phi$ is a rational fan over the lattice $\Lambda_W$ in $V$. It is also \emph{smooth}, in the sense that the primitive rays of each maximal cone form a basis of $\Lambda_W$.  Therefore, a discrete $\Phi$-submodular function $h: \R_\Phi \to \ZZ$ is the support function of a polytope whose vertices are integral over the dual lattice $\Lambda_{R^\vee}$, and conversely.
\end{proof}

%
%
%

\subsection{\textsf{The classical types: submodular, bisubmodular, disubmodular functions}}

For the classical root systems, these notions  are of particular combinatorial importance. Let us now describe them, keeping in mind that fundamental weights and their conjugates have simple combinatorial interpretations, as explained in Example \ref{ex:Wweights}.

\medskip

\noindent 1. (Type $A$: submodular functions) 
For $f: \R_{A_{d-1}} \rightarrow \RR$, let us write $f(S) := f(\e_S)$ for $\emptyset \subsetneq S \subsetneq [d]$ and $f(\emptyset) = f([d])=0$.
The $A_{d-1}$-submodular inequalities of Theorem \ref{thm:main}  say

\medskip

\begin{tabular}{rll}
local: & $f(S a) + f(S  b) \geq f(S) + f(S  a  b)$ & for $S \subseteq [d], \, \,   \{a, b\} \subseteq [d] - S$ \\
\\
global: &  $f(S) + f(T) \geq f(S \cap T) + f(S \cup T)$ & for $S, T \subseteq [d]$ 
\end{tabular}

\medskip
\noindent 
where for simplicity we omit brackets, for instance, denoting $Sab := S \cup\{a,b\}$.

The only difference with the classical notion of \emph{submodular functions} is the additional condition that $f([d]) = 0$. In fact, the submodular functions $F: 2^{[d]} \rightarrow \RR$ are precisely those of the form $F(S) = f(\e_S) + \alpha|S|$ for a $\Phi$-submodular function $f$ and a constant $\alpha$. Geometrically, we go from $F$ to $f$ by  translating the generalized permutohedron along the $\1$ direction so that it lies on the hyperplane $x_1 + \cdots + x_d = 0$.

\medskip

\noindent 2. (Type $B$ and $C$: bisubmodular functions) 
The submodular inequalities of type $B_d$ and $C_d$ are equivalent since they correspond to the same fan; we  focus on $C_d$. For $f: \R_{C_d} \rightarrow \RR$, let us write $f(S) = f(e_S)$ for any admissible $S \sqsubseteq [\pm d]$. The $C_d$-submodular inequalities of Theorem \ref{thm:main}  say

\medskip

\begin{tabular}{rll}
local: & $f(S  a) + f(S  b) \geq f(S) + f(S  a  b)$  
 & for $S \sqsubseteq [\pm d], \, \, |S| \leq d-2, \,\,   \{a, b\} \sqsubset [\pm d] - S$ \\
& $f(S  a) + f(S  \overline{a}) \geq 2f(S)$  
 & for $S \sqsubseteq [\pm d], \, \, |S| =d-1, \,\,   \{a\} \sqsubset [\pm d] - S$ \\
\\
 global: &  $f(S) + f(T) \geq f(S \sqcap T) + f(S \sqcup T)$ & for $S, T \sqsubseteq [d]$ 
\end{tabular}
\medskip

\noindent where $S \sqcap T = S \cap T$ and $S \sqcup T = \{e \in S \cup T \, : \, -e \notin S \cup T\}$ are admissible.
As Arcila showed in \cite{Arcila}, this is precisely the classical notion of \emph{bisubmodular functions} from optimization due to Fujishige \cite{Fujishige}.

\medskip

\noindent 3. (Type $D$: disubmodular functions) 
For $f: \R_{D_d} \rightarrow \RR$, let us write $f(S) = f(e_S)$ for any admissible $S \sqsubseteq [\pm d]$ of size at most $d-2$, and $g(S) = f(\frac12 e_S)$ for any admissible $S \sqsubseteq [\pm d]$ of size $d$.
The local $D_d$-submodular inequalities of Theorem \ref{thm:main}  say that for any admissible $S \sqsubseteq [\pm d]$

\medskip

\begin{tabular}{ll}
$f(S  a) + f(S  b) \geq f(S) + f(S  a  b)$ &  for 
$|S| \leq d-4,  \, \,   \{a, b\}  \sqsubseteq [\pm d] - S,$ \\
$f(S  a) + f(S b) \geq f(S) + g(S abc) + g(S ab\overline{c}) $ &  for 
$|S| = d-3,  \, \,   \{a, b, c\} \sqsubseteq [\pm d] - S$  \\
$g(S  a  b) + g(S \overline{ab}) \geq f(S) $ & for 
$|S| = d-2,  \, \,   \{a, b\} \sqsubseteq [\pm d] - S$  
\end{tabular}
\medskip

The global $D_d$-submodular inequalities can similarly be derived in a case-by-case analysis. It is easier to notice that a function that is piecewise linear on the Coxeter arrangement $D_d$ is also piecewise linear on the Coxeter arrangement $B_d$, where its convexity can be checked more cleanly. Accordingly, if $f$ and $g$ are defined on the admissible subsets of $[d]$ sizes at most $d-2$ and equal to $d$, respectively, define $h$ on all admissible subsets by
\[
h(S) = \begin{cases}
f(S) & \textrm{ if } |S| \leq d-2 \\
g(Sa) + g(S\overline{a}) & \textrm{ if } |S| = d-1 \textrm{ and } a \notin S \\
2g(S) & \textrm{ if } |S| = d
\end{cases}
\]
Then $(f,g)$ is disubmodular if and only if $h$ is bisubmodular; that is,
\[
h(S) + h(T) \geq h(S \sqcap T) + h(S \sqcup T) \qquad \textrm{ for } S, T \sqsubseteq [d].
\]

\noindent This seems to be a new notion, which we call \emph{disubmodular function}. We expect it to be useful in combinatorial optimization problems with underlying symmetry of type $D$.

\medskip

\noindent 4. (Exceptional types)
It would be very interesting to find applications of these notions for the exceptional Coxeter groups. For instance, might submodular functions of type $E$ shed new light on problems with an underlying symmetry of type $E_6, E_7,$ or $E_8$?

\section{\textsf{The symmetric case: weight polytopes and the inverse Cartan matrix }}\label{sec:symmetric}

The action of the Weyl group $W$ on the Coxeter complex naturally gives rise to actions of $W$ on the vector space $\PL(\Sigma_\Phi)$ and the deformation cone $\df(\Sigma_\Phi) \subset \PL(\Sigma_\Phi)$. 
This section is devoted to studying the deformations of the Coxeter permutohedron and the Coxeter submodular functions that are invariant under this action.

\subsection{\textsf{Weight polytopes}}\label{sec:weight}

Recall that the weight polytope $P_\Phi(x)$ of a point $x \in V$ is 
\[
P_\Phi(x) := \operatorname{conv}\{w\cdot x : w\in W\}.
\]
These are precisely the generalized Coxeter permutohedra that are invariant under the action of the Coxeter group. 
In this section we study them in more detail, collecting some properties that will play an important role in what follows.

\begin{defi}
The \emph{fundamental weight polytopes} or \emph{$\Phi$-hypersimplices} of the root system $\Phi$ are the $d$ weight polytopes $P_\Phi(\myomega_1), \ldots, P_\Phi(\myomega_d)$ corresponding to the fundamental weights of $\Phi$.
\end{defi}

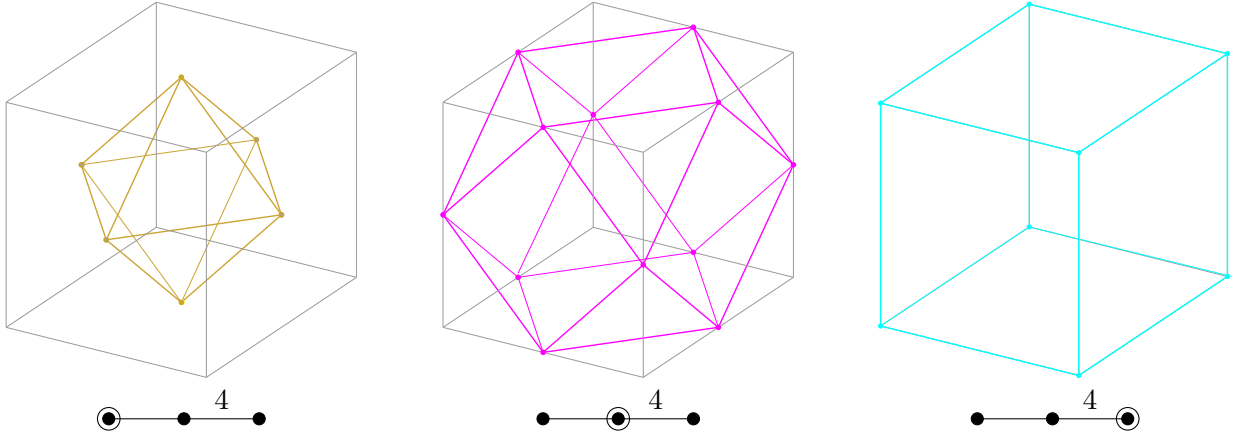
\begin{figure}[h]
\begin{subfigure}[b]{0.3\textwidth}
\centering
\includegraphics[height=5cm]{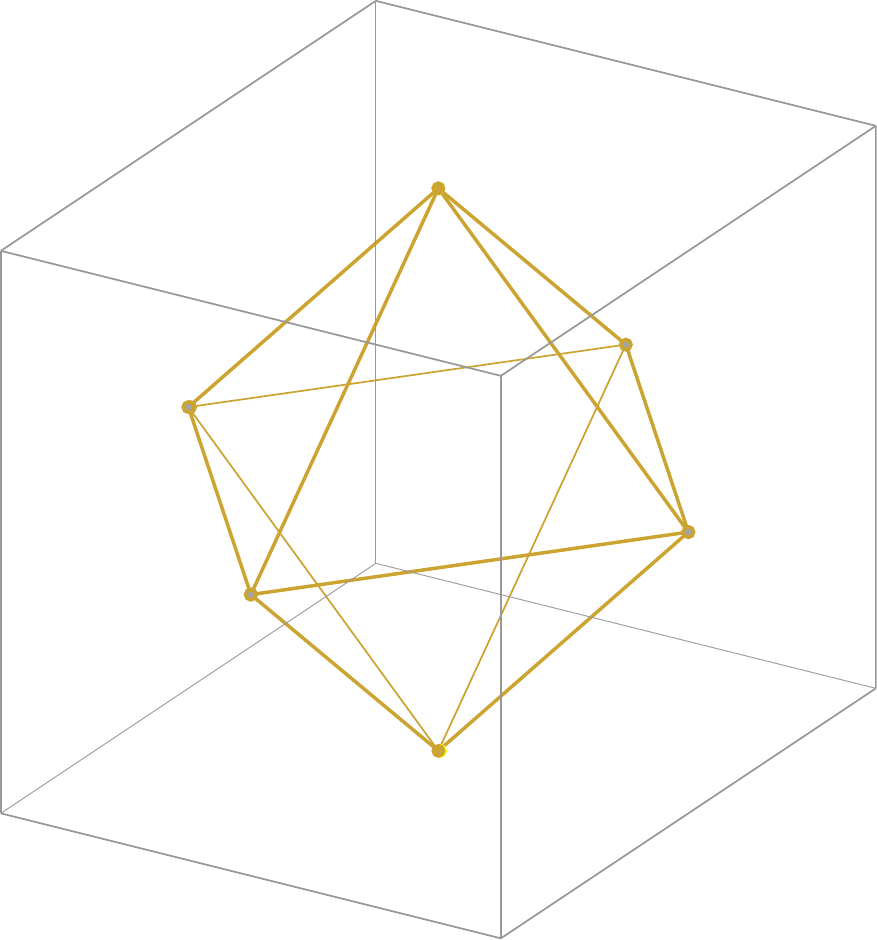}\qquad\qquad
\begin{tikzpicture}
\draw[fill]
(0,0) circle [radius = .08]
(1,0) circle [radius = .08]
(2,0) circle [radius = .08]
;
\draw
(0,0) -- (1,0) -- (2,0) node [midway, above] {4}
;
\draw (0,0) circle [radius = .15];
\end{tikzpicture}
\end{subfigure}
\qquad
\begin{subfigure}[b]{0.3\textwidth}
\centering
\includegraphics[height=5cm]{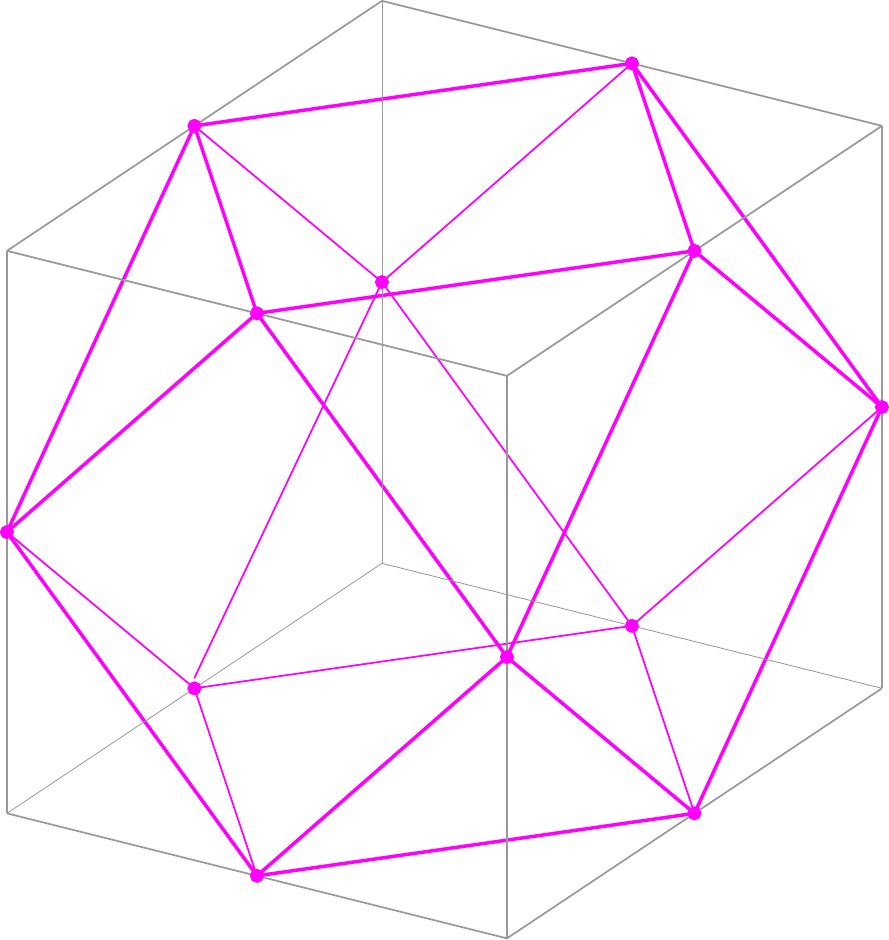}\qquad\qquad
\begin{tikzpicture}
\draw[fill]
(0,0) circle [radius = .08]
(1,0) circle [radius = .08]
(2,0) circle [radius = .08]
;
\draw
(0,0) -- (1,0) -- (2,0) node [midway, above] {4}
;
\draw (1,0) circle [radius = .15];
\end{tikzpicture}
\end{subfigure}
\qquad
\begin{subfigure}[b]{0.3\textwidth}
\centering
\includegraphics[height=5cm]{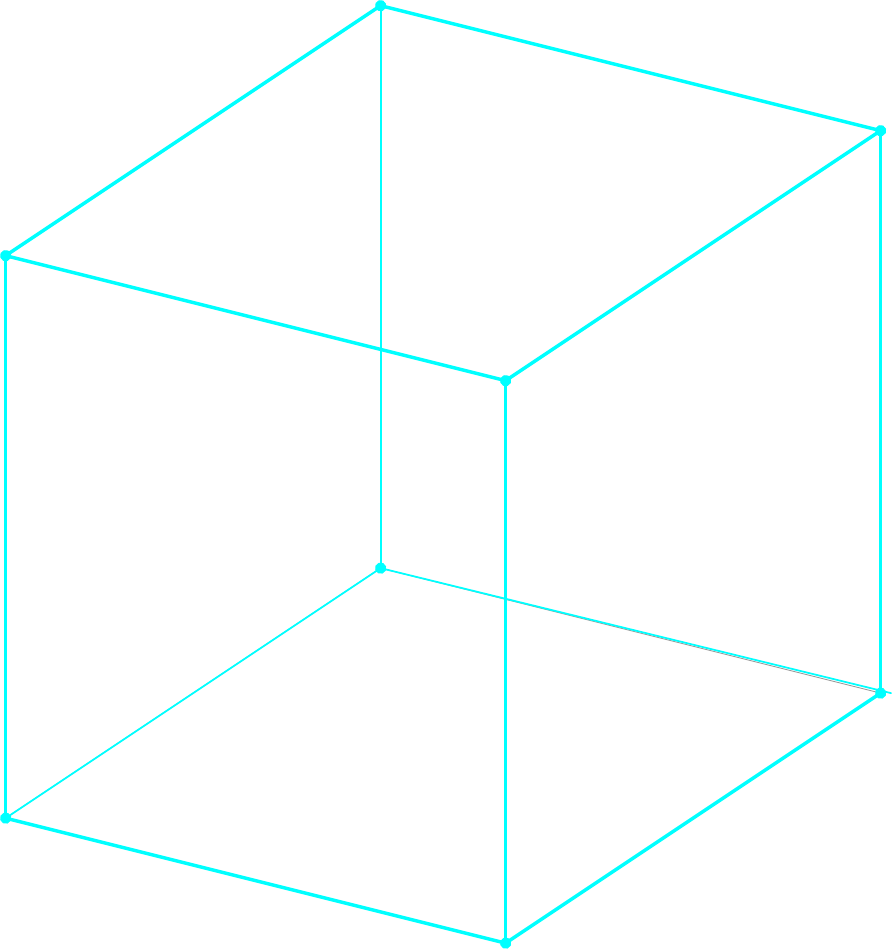}
\begin{tikzpicture}
\draw[fill]
(0,0) circle [radius = .08]
(1,0) circle [radius = .08]
(2,0) circle [radius = .08]
;
\draw
(0,0) -- (1,0) -- (2,0) node [midway, above] {4}
;
\draw (2,0) circle [radius = .15];
\end{tikzpicture}
\end{subfigure}

%

\caption{The fundamental weight polytopes $P_{C_3}(\lambda_i)$ of $C_3$ $(1 \leq i \leq 3)$; compare with Figure \ref{fig:C3}.b. \label{fig:C3hypersimplices}}
\end{figure}

The fundamental weight polytopes for the root system $C_3$ are shown in Figure \ref{fig:C3hypersimplices}.

Since $W$ acts transitively on the chambers of the Coxeter complex $\Sigma_\Phi$, in the study of the weight polytopes $P_\Phi(x)$ it is sufficient to consider only points $x$ in the fundamental domain $D$. For those points, the combinatorial type of the weight polytope $P_{\Phi}(x)$ is determined by the face of $D$ containing $x$ in its interior:

\begin{prop}\cite[\S1.12]{reflection} \label{prop:orbit} 
For $x$ in the interior of $C_I$, the chambers of the normal fan of $P_\Phi(x)$ are in bijection with $W/W_I$. The chamber of $\Sigma_{P_\Phi(x)}$ corresponding to the coset $wW_I$ is the union of the $|W_I|$ chambers of  the Coxeter complex $\Sigma_\Phi$ labeled $ww_I$ for $w_I \in W_I$. 
%
\end{prop}

The following observations about weight polytopes $P_\Phi(x)$ will be important to us.

\begin{cor} \label{cor:specialweightpolytopes}

\begin{enumerate}
\item
When $x$ is half the sum of the positive roots, $P_\Phi(x)$ is precisely the standard $\Phi$-permutohedron $\Pi_\Phi$. 

\item
When $x$ is in the interior of the fundamental chamber $D$, $P_\Phi(x)$ is normally equivalent to $\Pi_\Phi$. 
%
%

\item
When $x$ is in the interior of face $C_{[d]- I}$ of $D$, the polytope $P_\Phi(x)$ has positive edge length on the edge between $w  x$ and $w s_i  x$ for each $w \in W$ and $i \in I$, and zero everywhere else. In other words, its normal fan is obtained from $\Sigma_\Phi$ by only keeping the walls between the chambers $w  D$ and $ws_i  D$ for each $w \in W$ and $i \in I$.\end{enumerate}
\end{cor}

\begin{proof}
1. This follows directly from the definitions. 

\smallskip

\noindent 
2. and 3. The normal fan of $P_{\Phi}(x)$ is obtained from the Coxeter complex $\Sigma_\Phi$ by keeping only the $W$-translates of the walls of the fundamental chamber $D$ that do not contain $x$; that is, the walls between chambers $ D$ and $s_i D$ for each $i\in I$.
\end{proof}

We can describe any weight polytope as a Minkowski sum of the fundamental weight polytopes:

\begin{prop}\label{prop:orbitsum}
Let $\myomega_1, \ldots, \myomega_d$ be a set of fundamental weights of $\Phi$ and $a_1, \ldots, a_d \geq 0$.  Then 
\[
P_{\Phi}\left(\sum_{i=1}^d a_i \myomega_i\right) = \sum_{i = 1}^d a_i P_\Phi(\myomega_i)
\]
In particular, for any $x$ is in the interior of $C_{[d]- I}$, the weight polytope $P_{\Phi}(x)$ is normally equivalent to the Minkowski sum $\sum_{i\in I} P_\Phi(\myomega_i) = P_\Phi(\myomega_I).$
\end{prop}

\begin{proof}
Let us first prove the second statement. By Corollary \ref{cor:specialweightpolytopes}, the normal fan $\Sigma$ of 
$P_{\Phi}(x)$ is the coarsest common refinement of the fans $\Sigma_i$ for $i\in I$, where $\Sigma_i = \Sigma_{P_\Phi(\myomega_i)}$ is obtained from $\Sigma_\Phi$ by only keeping the walls between chambers $w D$ and $ws_i  D$.  

Let us call the two polytopes in the equation $P$ and $Q$.
 For any $x \in D$, the $x$-maximal face of $P_\Phi(\myomega_i)$ is its vertex $\myomega_i$. Therefore the $x$-maximal face of $Q$ is $\sum a_i \myomega_i$, which is thus a vertex of $Q$. By $W$-symmetry, $w  (\sum a_i \myomega_i)$ is also a vertex of $Q$ for every $w \in W$. Since every vertex of $P$ is a vertex of $Q$, and $P$ and $Q$ are normally equivalent by the previous paragraph,  $P=Q$.
\end{proof}

The face structure of weight polytopes is well understood. It was first studied in \cite{Max89} using the idea of a ``shadow" introduced by Tits \cite{Tit74}, and was further studied in \cite{Ren09} in a different context, using linear algebraic monoids.  It is shown in \cite[\S4]{GM18} that the following result from \cite{Ren09}, which was originally stated only for crystallographic root systems, follows from the results of \cite{Max89} and hence holds for arbitrary finite root systems.

\begin{theorem}\label{thm:weightfaces} \cite[Corollary 1.3]{Ren09} (cf.\ \cite[Theorem 2.4.3]{GM18}) For $I\subseteq \Delta$, let $\lambda$ be an element in the relative interior of the cone $C_I$, so that the isotropy group of $\lambda$ is $W_I$.  There is a bijection
\[
\left\{
\begin{gathered}
\text{$W$-orbits of faces of} \\
\text{the weight polytope $P_\Phi(\lambda)$}
\end{gathered}
\right\}
\longleftrightarrow 
\left\{
\begin{aligned}[r]
J \subseteq \Delta \mid \text{no connected component} \\
\text{ of $\Gamma|_J$ is contained in $\Gamma|_I$}
\end{aligned}
\right\}
\]
where $\Gamma$ is the Dynkin diagram. 
Under this bijection, such a subset $J$ corresponds to the $W$-orbit of a $|J|$-dimensional face $F_J$ that is combinatorially equivalent to $P_{\Phi_J}(\sum_{i \in (\Delta- I)\cap J}\lambda_i)$.
\end{theorem}

\begin{figure}[h]
\centering
\begin{subfigure}[b]{0.15\textwidth}
\centering
\includegraphics[height=2cm]{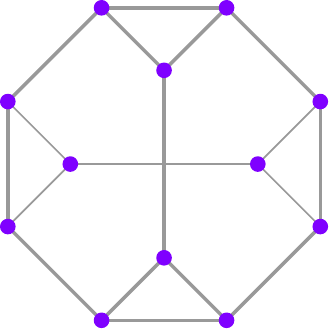} \\ \vspace{.3cm}
\begin{tikzpicture}
\node[color=white] at (0,0) [rectangle,draw] (){$\bullet$};
\draw[fill]
(0,0) circle [radius = .07]
(0.9,0) circle [radius = .07]
(1.8,0) circle [radius = .07]
;
\draw
(0,0) -- (0.9,0) -- (1.8,0); 
\draw (0,0) circle [radius = .15];
\draw (0.9,0) circle [radius = .15];
\end{tikzpicture}
\end{subfigure}
\begin{subfigure}[b]{0.15\textwidth}
\centering
\includegraphics[height=2cm]{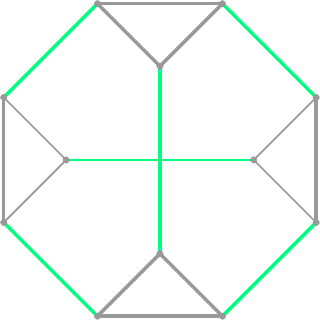} \\ \vspace{.3cm}

\begin{tikzpicture}
\draw[fill]
(0,0) circle [radius = .07]
(0.9,0) circle [radius = .07]
(1.8,0) circle [radius = .07]
;
\draw
(0,0) -- (0.9,0) -- (1.8,0); 
\draw (0,0) circle [radius = .15];
\draw (0.9,0) circle [radius = .15];
\draw [color=springgreen] (-0.2,-0.2) rectangle ++ (0.4,0.4);
\end{tikzpicture}
\end{subfigure}
\begin{subfigure}[b]{0.15\textwidth}
\centering
\includegraphics[height=2cm]{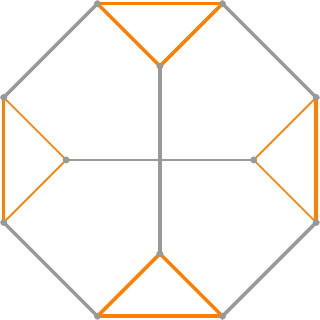} \\ \vspace{.3cm}

\begin{tikzpicture}
\draw[fill]
(0,0) circle [radius = .07]
(0.9,0) circle [radius = .07]
(1.8,0) circle [radius = .07]
;
\draw
(0,0) -- (0.9,0) -- (1.8,0); 
\draw (0,0) circle [radius = .15];
\draw (0.9,0) circle [radius = .15];
\draw [color=orange] (0.7,-0.2) rectangle ++ (0.4,0.4);
\end{tikzpicture}
\end{subfigure}
\begin{subfigure}[b]{0.15\textwidth}
\centering
\includegraphics[height=2cm]{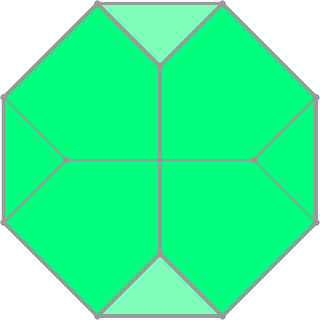} \\ \vspace{.3cm}
\begin{tikzpicture}
\draw[fill]
(0,0) circle [radius = .07]
(0.9,0) circle [radius = .07]
(1.8,0) circle [radius = .07]
;
\draw
(0,0) -- (0.9,0) -- (1.8,0); 
\draw (0,0) circle [radius = .15];
\draw (0.9,0) circle [radius = .15];
\draw [color=chartreuse] (-0.2,-0.2) rectangle ++ (1.3,0.4);
\end{tikzpicture}
\end{subfigure}
\begin{subfigure}[b]{0.15\textwidth}
\centering
\includegraphics[height=2cm]{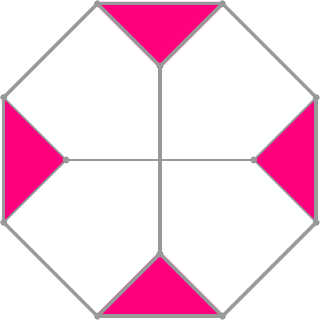} \\ \vspace{.3cm}

\begin{tikzpicture}
\draw[fill]
(0,0) circle [radius = .07]
(0.9,0) circle [radius = .07]
(1.8,0) circle [radius = .07]
;
\draw
(0,0) -- (0.9,0) -- (1.8,0); 
\draw (0,0) circle [radius = .15];
\draw (0.9,0) circle [radius = .15];
\draw [color=pink] (0.7,-0.2) rectangle ++ (1.3,0.4);
\end{tikzpicture}
\end{subfigure}
\begin{subfigure}[b]{0.15\textwidth}
\centering
\includegraphics[height=2cm]{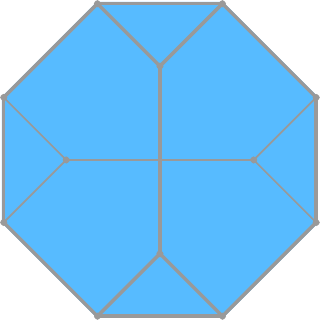} \\ \vspace{.3cm}

\begin{tikzpicture}
\draw[fill]
(0,0) circle [radius = .07]
(0.9,0) circle [radius = .07]
(1.8,0) circle [radius = .07]
;
\draw
(0,0) -- (0.9,0) -- (1.8,0); 
\draw (0,0) circle [radius = .15];
\draw (0.9,0) circle [radius = .15];
\draw [color=azure] (-0.2,-0.2) rectangle ++ (2.2,0.4);
\end{tikzpicture}
\end{subfigure}
\caption{\label{fig:faceorbits}The six face orbits for $P_{A_3}(\lambda_1+\lambda_2)$ correspond to the subsets $J \subseteq [3]$ not having $\{3\}$ as a connected component : the twelve vertices $(\emptyset)$, six of the edges $(\{1\})$, the remaining twelve edges $(\{2\})$, the four hexagons $(\{1,2\})$, the four triangles $(\{2,3\})$, and the whole polytope $(\{1,2,3\})$.} 
\end{figure}
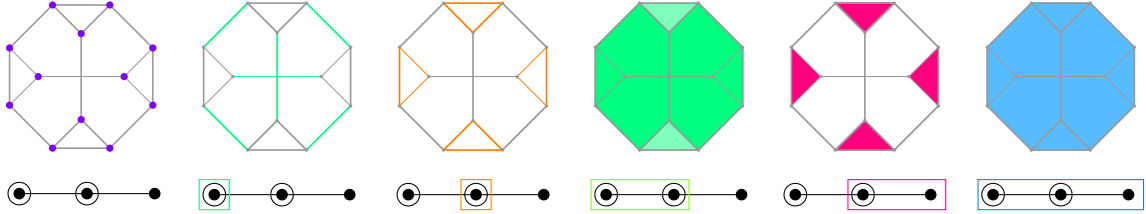

\begin{ex}
Let $\Phi=A_3$ and $I = \{3\}$. Consider the weight polytope $P=P_{A_3}(\lambda_1+\lambda_2)$ of $\lambda_1 + \lambda_2$, whose isotropy group is $W_{\{3\}}=\{e,s_3\}$. 
Theorem \ref{thm:weightfaces} asserts that there is an $S_4$--orbit of faces of $P$ for each subset $J \subseteq [3]$ such that $\{3\}$ does not contain a connected component of $\Gamma|_J$. These are shown in Figure \ref{fig:faceorbits}.
\end{ex}

\begin{cor}\label{cor:triangles}
The fundamental weight polytope $P_\Phi(\myomega_i)$ only has triangular 2-faces if and only if the edges adjacent to $i$ in the Dynkin diagram are unlabeled; that is, the Dynkin diagram 
$\Gamma|_{N(i) \cup i}$ 
is simply laced.
\end{cor}

\begin{proof}
By Theorem \ref{thm:weightfaces} applied to $I = \Delta - i$, each 2-face of $P_\Phi(\myomega_i)$ corresponds to a subset $J=\{i,j\}$ of $\Delta$ such that $i$ and $j$ are neighbors in $\Gamma$. That 2-face is combinatorially equivalent to the weight polytope $P_{\Phi_{\{i,j\}}}(\lambda_j)$, which is a triangle if and only if $m_{ij}=3$, as desired.
\end{proof}

\subsection{\textsf{Symmetric $\Phi$-submodular functions}}\label{sec:symmPhisubmodularcone}

Say a $\Phi$-submodular function $f$ is \emph{symmetric} if it is invariant under the action of the Weyl group; that is, if
\[
f(w \myomega_i) = f(\myomega_i) \qquad \textrm{ for all } w \in W \textrm{ and } 1 \leq i \leq d.
\]
These functions correspond to the support functions of the weight polytopes of Section \ref{sec:weight}. They form the \emph{symmetric $\Phi$-submodular cone}, a linear slice of $\textrm{Def}(\Sigma_\Phi)$.

By identifying a symmetric $\Phi$-submodular function with its values on the fundamental weights, we may think of this cone as living in $\RR^d$. We now show that this cone has an elegant description: it is the simplicial cone generated by the rows of the inverse of the Cartan matrix. 
This inverse matrix was first described by Lusztig and Tits \cite{LusztigTits}; an explicit list 
 is given in \cite{HumphreysLie, WeiZou}.

A key role is played by the fundamental weight polytopes of the previous subsection. The results are a bit more elegant if we  rescale them and work with the coweight polytopes instead.

\begin{lem} \label{lem:inverseCartan}
The $\Phi$-submodular function $h_k$ of the fundamental coweight polytope $P_\Phi(\myomega^\vee_k)$ is
\[
h_k(w \myomega_i) = A^{-1}_{ki} \qquad \textrm{ for } w \in W, \quad  1 \leq i \leq d,
\]
where $A^{-1}$ is  the inverse of the Cartan matrix. 
\end{lem}

\begin{proof}
Let $x$ be in the interior of the fundamental domain $D$. Since $\myomega_i \in D$, the $\myomega_i$-maximal face of $P_\Phi(\myomega^\vee_k)$ must contain the $x$-maximal face of $P_\Phi(\myomega^\vee_k)$, which is the vertex $\myomega_k$.
It follows that the $\myomega_i$-maximal value of $P_\Phi(x)$ is $h_k(\myomega_i) = \langle \myomega^\vee_k, \myomega_i\rangle = A^{-1}_{ki}$ by \eqref{Ainverse}. By $W$-symmetry, this is also the value of $h_k(w \myomega_i)$ for any $w \in W$. 
\end{proof}

\begin{prop}\label{prop:sym}
The symmetric $\Phi$-submodular cone is the simplicial cone generated by the rows of the inverse Cartan matrix of $\Phi$.
\end{prop}

\begin{proof}
Proposition \ref{prop:orbitsum} shows that any weight polytope is a Minkowski sum of the fundamental coweight polytopes $P_\Phi(\myomega^\vee_k)$.
Since the support function of a Minkowski sum $aP + bQ$ is given by
$h_{aP+bQ} = ah_P + bh_Q$ for $a, b \geq 0$, this means that any symmetric $\Phi$-submodular function is a non-negative combination of the functions described in Lemma \ref{lem:inverseCartan}.
Since $A^{-1}$ is invertible, these functions are linearly independent. The desired result follows.
\end{proof}

\section{\textsf{Facets of the $\Phi$-submodular cone}}\label{sec:facets}

In this section we describe and enumerate the facets of the $\Phi$-submodular cone. We first prove that all the wall crossing inequalities define facets; for an arbitrary polytope, this is rarely the case. This claim is equivalent to saying that all the rays spanned by the $I_{\tau}$s, as described in \eqref{ineq:wallcross}, are extremal in the Mori cone 
$\overline{NE}(\Sigma_\Phi) = \textrm{cone}(I_{\tau} \, : \, \tau \textrm{ is a wall of } \Sigma_\Phi)$ in $\left(\operatorname{PL}(\Sigma_\Phi)\right)^\vee$.

\begin{theorem}\label{thm:facet}
Every local $\Phi$-submodular inequality \eqref{ineq:localsubmod} is a facet of the $\Phi$-submodular cone. 
\end{theorem}

\begin{proof}
By Corollary \ref{cor:specialweightpolytopes}.3 we can produce, for each $1 \leq i \leq d$, a generalized $\Phi$-permutohedron $Q_i=P_{\Phi}(\myomega_i)$ whose normal fan is obtained from $\Sigma_\Phi$ by removing the walls $wH_i$ separating chambers $w  D$ and $ws_i  D$ for all $w\in W$. The support function of this polytope satisfies
\begin{eqnarray}
I_{\tau}(h_{Q_i}) &=& 0, \quad \textrm{ if } \tau= wH_i \textrm{ for some } w\in W, \textrm{ and }  \label{eq:trick}\\
I_{\tau}(h_{Q_i}) &>& 0, \quad \textrm{ otherwise}.
\end{eqnarray}
This means that the set of rays $\{I_{wH_i} \, : \, w\in W\}$ form a face $F_i$  of the Mori cone, so at least one of them must be extremal. But these rays form an orbit of the action of $W$ on the Mori cone, so if one of them is extremal, all are extremal.
\end{proof}


\begin{theorem}\label{thm:count}
The number of facets of the $\Phi$-submodular cone is 
\[
\displaystyle \sum_{i=1}^d \dfrac{|W|}{|W_{[d] -  N(i)}|},
\]
where $N(i)$ is the set of neighbors of $i$ in the Dynkin diagram. They come in $d$ symmetry classes up to the action of $W$.
For the classical root systems, these numbers are:
\begin{eqnarray*}
A_{d-1} &:& d(d-1)2^{d-3} \\
BC_d &:& 2d(d-1)3^{d-2} + d2^{d-1} \\
D_d &:& 2d(d-1)3^{d-2} - d(d-1)2^{d-2}
\end{eqnarray*}

\end{theorem}
\begin{proof}
We have one local $\Phi$-submodular inequality for each pair of an element $1 \leq i \leq d$ and a group element $w \in W$, but there are many repetitions. For each $i$ we now show that the set of elements $w$ stabilizing the wall-crossing inequality \eqref{ineq:wallcox} is $W_{[d]-N(i)}$.

If an element $w$ stabilizes \eqref{ineq:wallcox}, it must stabilize the support of the right hand side, that is, the set of fundamental weights $\{\myomega_j \, : \, j \in N(i)\}$. Therefore $w$ stabilizes the sum of those weights, which is in the interior of cone $C_{[d]-N(i)}$. By Proposition \ref{prop:stab}, $w \in W_{[d]-N(i)}$. 

Conversely, suppose $w \in W_{[d]-N(i)}$. Then for each $j \in N(i)$ we have $w \in W_{[d]-j}$, so $w$ stabilizes $\myomega_j$ individually. Therefore $w$ does stabilize the right hand side of \eqref{ineq:wallcox}. Now, each simple reflection $s_k$ with $k \notin [d]-N(i)-i$ stabilizes $\myomega_i$ because $k \neq i$, and hence it also stabilizes $s_i   \myomega_i$ since $s_i$ and $s_k$ commute. The remaining reflection $s_i$ interchanges $\myomega_i$ and $s_i  \myomega_i$. It follows that each generator of $W_{[d]-N(i)}$, and hence the whole parabolic subgroup, stabilizes the left-hand side of   \eqref{ineq:wallcox} as well.

We conclude that, for fixed $i$, each inequality in \eqref{ineq:wallcox} is repeated ${|W_{[d]-N(i)}|}$ times, and hence the number of different inequalities is ${|W|}/{|W_{[d]-N(i)}|}$. Furthermore, there is one symmetry class of inequalities for each $i$. The desired result follows. 

One may then compute explicitly the number of facets for the classical root systems, using that $|W_{A_{r-1}}| = r!$, $|W_{B_r}| = 2^rr!,$ and $|W_{D_r}| = 2^{r-1}r!$. 
\end{proof}

Notice that if $r$ and $r'$ are rays and $C$ and $C'$ are adjacent chambers of the Coxeter complex such that $r$ belongs to $C-C'$ and $r'$ belongs to $C'-C$, then the linear relation between the rays of $C$ and $C'$ 
is determined entirely by the rays $r$ and $r'$, independently of the choice of chambers $C$ and $C'$. This offers an explanation for the repetition of the wall-crossing inequalities. This property, which significantly simplifies the study of the deformation cone, holds for several interesting combinatorial fans; for instance, $g$-vector fans of Coxeter associahedra and normal fans of graph associahedra. \cite{Pilaud1, Pilaud2}

\section{\textsf{Some extremal rays of the $\Phi$-submodular cone}}\label{sec:rays}

On the opposite end of the facets, we now discuss the problem of describing the extremal rays of the $\Phi$-submodular cone $\operatorname{Nef}(\Sigma_\Phi)$.  These rays correspond to indeformable generalized $\Phi$-permutohedra.

\begin{defi}
We say a polytope $P$ is \emph{indecomposable} or \emph{indeformable} if its only deformations, in the sense of Definition \ref{defi:defo}, are its multiples (up to translation) \cite{shep}; that is, if its nef cone $\operatorname{Nef}(P)$ is a single ray.
\end{defi}


Describing \textbf{all} the extremal rays of $\operatorname{Nef}(\Sigma_\Phi)$ seems to be a very difficult task, even in the classical case $\Phi = A_{d-1}$. For example, the matroid polytope $P_M$ of any connected matroid $M$ on $[d]$ is a ray of $\operatorname{Nef}(\Sigma_\Phi)$. \cite{extreme, econ}. Therefore the number of rays of this nef cone is doubly exponential, because the asymptotic proportion of matroids that are connected is at least $1/2$ and conjecturally equal to $1$ \cite{Mayhewetal} and the number $m_d$ of matroids on $[d]$ satisfies $\log \log m_d \geq d - \frac32 \log d - O(1)$. \cite{Knuth} The cone $\operatorname{Nef}(\Sigma_{A_{d-1}})$ has been computed for $d \leq 5$; for $d=5$ it has only $80$ facets in six $S_5$ symmetry classes, while it has $117978$ rays in $1319$ $S_5$ symmetry classes. \cite{Mortonetal, Studeny}

We focus here on the more modest task of describing some interesting families of rays; \emph{i.e.}, indecomposable generalized $\Phi$-permutohedra.
Our main tools will be the following simple sufficiency criterion.
%
%
%

\begin{prop}\label{prop:alltriangles} \cite{shep}
If all $2$-faces of a polytope $P$ are triangles then $P$ is indecomposable.
\end{prop}

We will also use the following computational tool:

\begin{rem}\label{rem:computational} 
To check computationally whether a polytope $P$ is indecomposable, one could in principle ``simply" compute the dimension of its deformation cone. Unfortunately, this is not easy to do in practice. When $P$ is a deformation of the $\Phi$-permutohedron $\Pi_\Phi$ (or some other polytope with a nice deformation cone) and we know its support function $h_P$, there is a shortcut available to us. Since Def$(P)$ is the intersection of Def$(\Pi_\Phi)$ with the facet-defining hyperplanes that contain $h_P$, we can now determine which wall-crossing inequalities \eqref{ineq:localsubmod} $h_P$ satisfies \textbf{with equality}. If, after modding out by globally linear functions, those wall-crossing equalities cut out a 1-dimensional subspace, then Def$(P)$ is just a ray, and the polytope $P$ is indecomposable.
\end{rem}

%
%

The following is our main result about rays of the $\Phi$-submodular cone.
Recall that $N(i)$ denotes the set of nodes in the Dynkin diagram $\Gamma(\Phi)$ adjacent to the node $i$.

\begin{theorem}\label{thm:orbit}
A weight polytope $P$ of a crystallographic root system $\Phi$ is indecomposable if and only if $P = k P_\Phi(\myomega_i)$ for $k>0$ and a fundamental weight $\myomega_i$ such that the edges adjacent to $i$ in the Dynkin diagram are unlabeled; that is, the Dynkin diagram $\Gamma(\Phi_{N(i) \cup i})$ is simply laced.  
\end{theorem}

\begin{proof}
Proposition \ref{prop:orbitsum} shows that if a weight polytope is indecomposable, it must be a multiple of $P_\Phi(\myomega_i)$ for some fundamental weight $\myomega_i$.  When the Dynkin diagram $\Gamma(N(i) \cup i)$ is simply laced, we showed in Corollary \ref{cor:triangles}
that all the $2$-faces of $P_\Phi(\myomega_i)$ are triangles, so this polytope is indecomposable by Proposition \ref{prop:alltriangles}.

To show that all other fundamental weight polytopes are decomposable, we do a case by case analysis through the classification \ref{thm:classification}.  The Dynkin diagrams that have nodes $i$ such that $N(i)\cup i$ has an edge with label greater than 3 are $B_d, C_d, F_4,$ and $G_2$ Only the types $B_d$ and $C_d$ provide infinite families of weight polytopes, so we prove our claim in these two cases. One checks the remaining cases $F_4, G_2$ individually.

\smallskip

$\Phi = B_d$ or $C_d$ and $i=d$: In this case $\myomega_d$ equals $\frac12(e_1 +e_2 + \cdots +e_d)$ and $e_1 +e_2 + \cdots +e_d$, respectively. In both cases the orbit polytope is a hypercube of dimension $d$, which is the Minkowski sum of $d$ lines, and hence decomposable. 

\smallskip

$\Phi = B_d$ or $C_d$ and $i=d-1$: We have $\myomega_{d-1} = e_1+e_2+\cdots+ e_{d-1}$ and the weight polytope is the same in types $B$ and $C$. Now we claim that the orbit of $e_1 + \cdots + e_{d-1}$ under the action of $W_{C_d}$ is the same as its orbit under the action of $W_{D_d}$. To see this, recall that $W_{C_d}$ acts by all permutations and sign changes of the coordinates, while $W_{D_d} \leq W_{C_d}$ consists of those actions where the number of sign changes is even. Therefore the $W_{C_d}$-orbit of $\myomega_{d-1}=e_1 + \cdots + e_{d-1}$ consists of the vectors $v=w\myomega_{d-1}$ with one coordinate equal to $0$ and all other coordinates equal to $1$ or $-1$. By adding a sign change to $w$ in the $0$ coordinate if needed, we can arrange for it to be an element of $D_d$, as desired.

This observation, combined with Proposition \ref{prop:orbitsum}, tells us that
\begin{eqnarray*}
P_{C_d}(e_1 + \cdots + e_{d-1}) &=&  P_{D_d}(e_1 + \cdots + e_{d-1}) \\
&=&
P_{D_d}\left(\frac{e_1 + \cdots + e_{d-1}-e_d}2\right) +
P_{D_d}\left(\frac{e_1 + \cdots + e_{d-1}+e_d}2\right) 
\end{eqnarray*}
keeping in mind that $\frac12(e_1 + \cdots + e_{d-1}-e_d)$ and $\frac12(e_1 + \cdots + e_{d-1}+e_d)$ are the last two fundamental weights of type $D$. These two polytopes are deformations of the $D_d$-permutohedron, which is itself a deformation of the $C_d$-permutohedron. Therefore the fundamental weight polytope $P_{C_d}(\myomega_{d-1})$ is decomposable in this case as well.
\end{proof}

\begin{rem}
By Proposition \ref{prop:alltriangles}, any face of the indecomposable weight polytopes is also indecomposable. These are also rays of the nef cone by Corollary \ref{cor:zonofaces}: in types $A_n, BC_n, D_n$, we get exponentially many such rays of the nef cone as a function of $n$.
\end{rem}


\begin{rem}
Theorem \ref{thm:orbit} can fail for non-crystallographic root systems. More precisely, it fails for the fundamental weight polytopes $P_{H_3}(\myomega_2)$ and $P_{H_4}(\myomega_3)$, which are indecomposable. We have verified this by computer as outlined in Remark \ref{rem:computational}.\footnote{The supporting files are available at \texttt{http://math.sfsu.edu/federico/Articles/deformations.html}.} 
By Corollary \ref{cor:specialweightpolytopes}.3, the support function $h_k$ for $P_{\Phi}(\myomega_k)$ lies precisely on the facet hyperplanes given by local $\Phi$-submodular conditions \ref{ineq:localsubmod} with $i \neq k$.
In each of these two cases, those hyperplanes intersect in a line, making the nef cone of $P_{\Phi}(\myomega_k)$ one-dimensional. 
\end{rem}

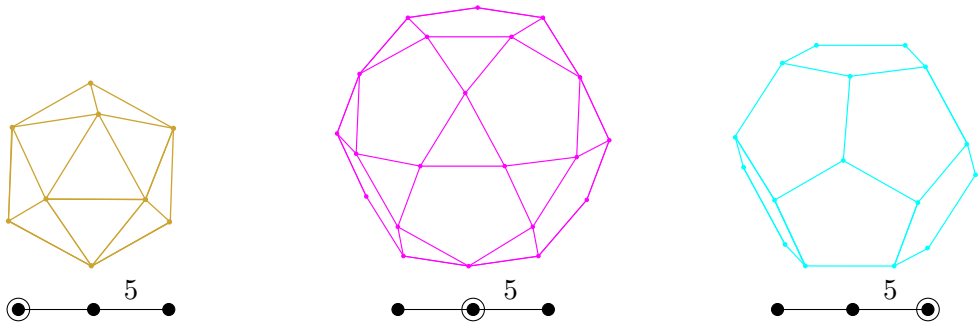
\begin{figure}[H]
\centering

\begin{subfigure}[b]{0.3\textwidth}
\centering
\includegraphics[height=25mm]{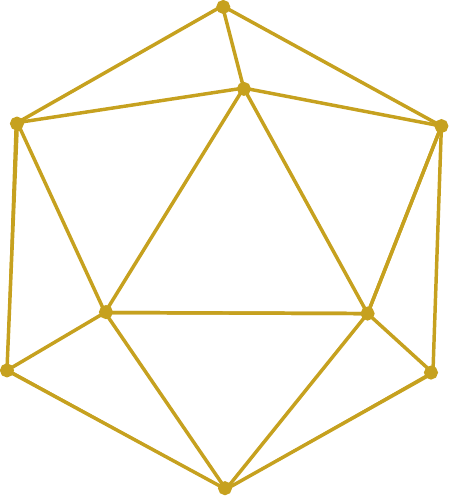}\qquad\qquad
\begin{tikzpicture}
\draw[fill]
(0,0) circle [radius = .08]
(1,0) circle [radius = .08]
(2,0) circle [radius = .08]
;
\draw
(0,0) -- (1,0) -- (2,0) node [midway, above] {5}
;
\draw (0,0) circle [radius = .15];
\end{tikzpicture}
\end{subfigure}
\begin{subfigure}[b]{0.3\textwidth}
\centering
\includegraphics[height=35mm]{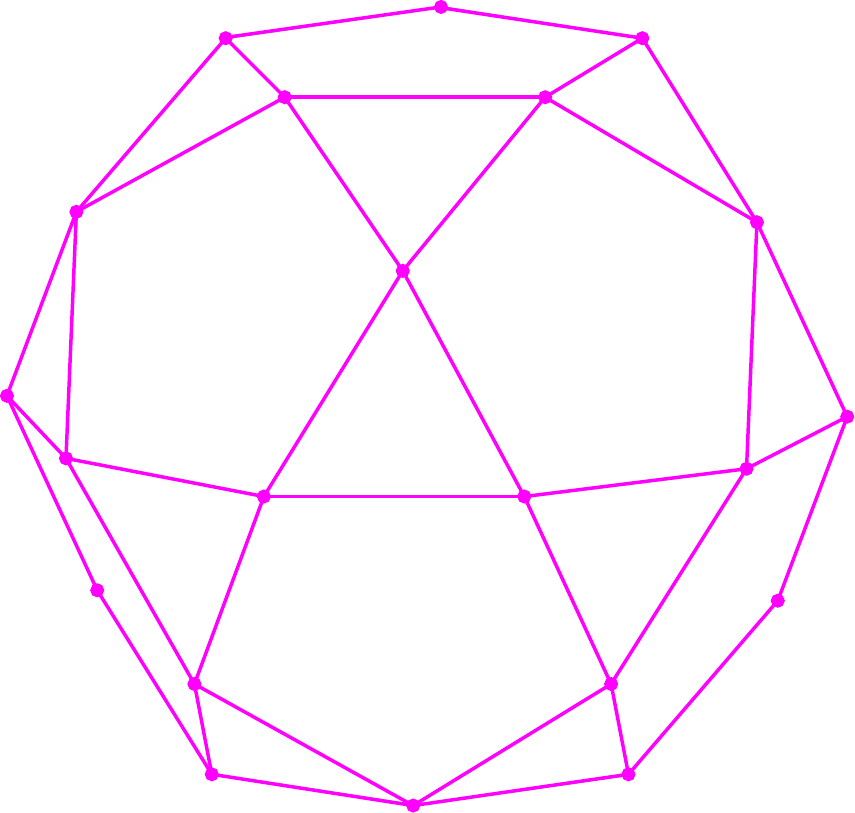}\qquad\qquad
\begin{tikzpicture}
\draw[fill]
(0,0) circle [radius = .08]
(1,0) circle [radius = .08]
(2,0) circle [radius = .08]
;
\draw
(0,0) -- (1,0) -- (2,0) node [midway, above] {5}
;
\draw (1,0) circle [radius = .15];
\end{tikzpicture}
\end{subfigure}
\begin{subfigure}[b]{0.3\textwidth}
\centering
\includegraphics[height=30mm]{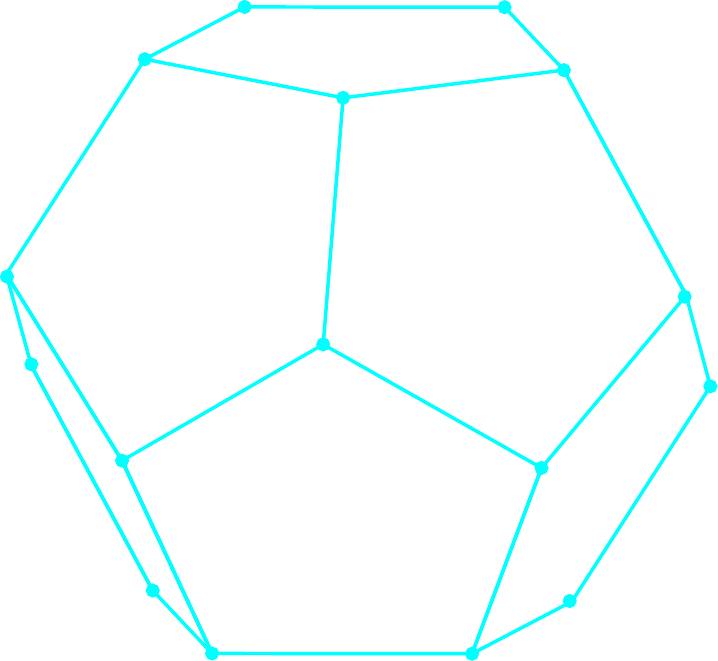}
\begin{tikzpicture}
\draw[fill]
(0,0) circle [radius = .08]
(1,0) circle [radius = .08]
(2,0) circle [radius = .08]
;
\draw
(0,0) -- (1,0) -- (2,0) node [midway, above] {5}
;
\draw (2,0) circle [radius = .15];
\end{tikzpicture}
\end{subfigure}
\qquad
\caption{The icosahedron $P_{H_3}(\myomega_1)$ is indecomposable because its $2$-faces are triangles. A computation shows that the icosidodecahedron $P_{H_3}(\myomega_2)$ is also indecomposable. The dodecahedron $P_{H_3}(\myomega_3)$ is decomposable because we can push away one of its pentagonal faces.} 
\end{figure}

Let us verify Theorem \ref{thm:orbit} for a few examples of interest.

\medskip

\noindent 1. (Type A)
The fundamental weight polytope $P_{A_d}(\myomega_i)$ is the  hypersimplex $\Delta(i,d+1) = \conv(e_S \, : \, S \subseteq [d+1], |S|=i)$ which only has triangular $2$-faces, and hence is indecomposable.

\medskip

\noindent 2. (Type BC)
In type $C_2$ the fundamental weight polytopes are the diamond and the square, which are indeed decomposable. This is consistent with the fact that the Dynkin diagram has no node satisfying the condition of Theorem \ref{thm:orbit}. 
In type $C_3$, they are shown in Figure \ref{fig:C3hypersimplices}. The octahedron is indeed indecomposable, while the rhombic dodecahedron is the Minkowski sum of two tetrahedra in opposite orientations, and the cube is the Minkowski sum of three segments.

%

%
%
%
%

\section{\textsf{Further questions and future directions.}}\label{sec:future}

\begin{enumerate}

\item
In type $A$, generalized permutohedra have the algebraic structure of a Hopf monoid; in fact, they are the universal family of polytopes that support such a structure \cite{aa}.  This leads to numerous interesting algebraic and combinatorial consequences. A crucial observation that makes this work is that for any generalized permutohedron $P$ in $\RR^E$ and any subset $\emptyset \subsetneq S \subsetneq E$, the maximal face of $P$ in direction $e_S$ decomposes naturally as the product of two generalized permutohedra in $\RR^S$ and $\RR^{E-S}$, respectively.

One of the main motivations for this project was the expectation that, similarly, generalized $\Phi$-permutohedra should be an important example of a new kind of algebraic structure: a \emph{Coxeter Hopf monoid} \cite{Rodriguez}.  It is still true that if $P$ is a generalized $\Phi$-permutohedron and $r=w\myomega_i$ is a ray, the maximal face of $P$ in direction $r$ is a generalized $\Phi_{[n]-i}$--permutohedron. It decomposes as a product of one, two, or three generalized Coxeter permutohedra, depending on the number of neighbors of $i$ in the Dynkin diagram. We plan to  further develop this algebraic structure in an upcoming paper.

\item
In the classical types $A_n$ and $BC_n$, the notions of $\Phi$-submodular functions correspond to submodular and bisubmodular functions, which are well studied in optimization \cite{Fujishige, Murota, Schrijver}. We expect that disubmodular functions should play a similar role in combinatorial optimization problems with an underlying symmetry of type $D$. Similarly, it would be very interesting to find applications for the exceptional $\Phi$-submodular functions, for instance, to problems with an underlying symmetry of type $E_6, E_7$, or $E_8$.

\item
In type $A$, 
every generalized permutohedron in $\RR^d$ is a signed Minkowski sum of the simplices $\Delta_S = \conv(e_s \, : \, s \in S)$ for $\emptyset \subsetneq S \subseteq [d]$. 
 Geometrically, this means that there is a nice choice of $(2^d-1)$ rays of the $(2^d-1)$-dimensional submodular cone -- namely, the rays corresponding to the polytopes $\Delta_S$ -- which also forms  a basis for $\RR^{2^d-1}$. Remarkably, since one may compute the mixed volumes of these polytopes $P_S$, one obtains combinatorial formulas for the volume of any generalized permutohedron.  For details, see \cite{matroidvolumes, Postnikov}.
 
Is there a similarly nice choice of $|\R_\Phi|$ rays of the $\Phi$-submodular cone that generate all others? Can one compute their mixed volumes? If so, one would obtain a formula for the volume of an arbitrary generalized $\Phi$-permutohedron. In type $A$, the $2^d-1$ non-empty faces of the simplex $P_{A_{d-1}}(\myomega_1)$ suffice, as explained above. 
Unfortunately (but still interestingly), the $3^d-1$ non-empty faces of the cross-polytope $P_{B_d}(\myomega_1)$, which are rays of Def$(B_d)$, only span a subspace of dimension $\frac12(3^d-(-1)^d)$ of $\RR^{3^d-1}$ \cite{doker}.  Can one do better, either in type $B$ or in general? For some related work on the mixed volumes of the fundamental weight polytopes, see \cite{Babson, Croitoru, Liu, Postnikov}.

\item
The framework presented here makes it very natural to define the \emph{rank function} of a Coxeter matroid $M$ of type $\Phi$ to be the support function $h_{Q(M)}:\RR \rightarrow R$ of its Coxeter matroid polytope. It would be interesting and useful to give a characterization of these rank functions.

\item
Is there a good characterization of the indecomposable Coxeter matroids?
This has been done beautifully in type $A$ \cite{extreme, econ}:
a matroid polytope $Q(M)$ is indecomposable if and only if, upon deleting all loops and coloops, the matroid $M$ is connected. Equivalently, for a rank $r$ matroid $M$ on $[d]$, the matroid polytope $Q(M)$ is indecomposable if and only if it is a full-dimensional subset of the hypersimplex $\Delta(d, r)$.

The analogous statement does not hold for Coxeter matroids, even when one accounts for the fact that, unlike in type $A$, some fundamental weight polytopes can be decomposable. 
For example, consider the polytope highlighted below; it is a full-dimensional subset of the icosahedron -- an indecomposable fundamental weight polytope of type $H_3$. However, it is decomposable, since one can deform it by shortening the four middle edges until the two short ones disappear. 

\begin{figure}[H]
\centering
\includegraphics[height=45mm]{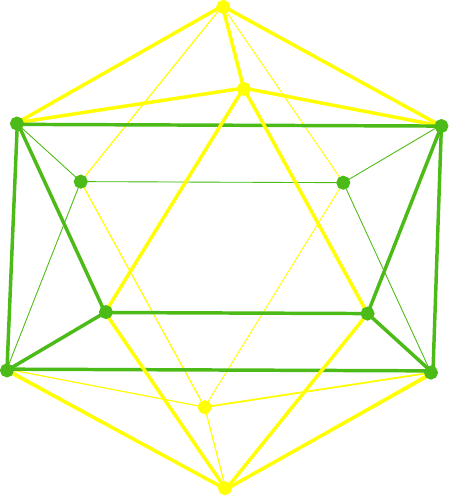}
\end{figure}

\end{enumerate}

\section{\textsf{Acknowledgments.}} We thank David Arcila, Jeff Doker, Andr\'es Rodr\'{\i}guez Rey, and Milan Studen\`y for valuable conversations about this project over the years. We are also grateful to the referees for their careful reading of the paper and their valuable suggestions; in particular, for pointing us to Theorem \ref{thm:weightfaces}.

Part of this work was done while FA, FC, and AP were visiting the Mathematical Sciences Research Institute in Berkeley in Fall 2017 and while FA was visiting the Simons Institute for the Theory of Computing in Spring 2019. We thank them for offering a wonderful setting to do mathematics, and the NSF and Simons Foundation for their financial support of these visits.

\footnotesize

\bibliography{biblio}

\bibliographystyle{plain}

\end{document}